\documentclass[11pt]{article}
\usepackage[frenchb,english]{babel}
\usepackage[utf8]{inputenc}
\usepackage[T1]{fontenc} 
\usepackage{amsmath}
\usepackage{amssymb}
\usepackage{amsthm}
\usepackage{color}
\usepackage{caption}
\usepackage[pdftex]{graphicx}
\usepackage{epstopdf}
\usepackage{pstricks}
\usepackage{subfigure}
\usepackage{tikz}
\usetikzlibrary{patterns}

\usepackage{cases}

\newtheorem{theorem}{Theorem}[section] 
\newtheorem{remark}[theorem]{Remark} 
\newtheorem{prop}[theorem]{Proposition} 
 
\newtheorem{defn}[theorem]{Definition} 

\numberwithin{equation}{section}

\renewcommand{\d}[1]{\,{\rm d}#1}
\renewcommand{\epsilon}{\varepsilon}

\newcommand{\pd}[2]{\partial_{#2} #1}

\newcommand{\jump}[1]{\left[\!\left[ #1 \right]\!\right]} 
\newcommand{\average}[1]{\left\{ #1 \right\}}

\def\RR{\mathbb{R}}

\let\epsilon=\varepsilon
\let\phi=\varphi

\title{
  Towards a new friction model for shallow water equations \\
  through an interactive viscous layer}
\date{\today}

\begin{document}
\maketitle

\begin{center}
  Fran\c{c}ois James$^{*}$, Pierre-Yves Lagr\'ee$^{o}$, Minh H. Le$^{\dag}$ and Mathilde Legrand$^{*}$
  \\
  \footnotesize{
    $^o$Sorbonne Université, CNRS, UMR 7190, \\ Institut Jean Le Rond d'Alembert, F-75005 Paris, France\\
    $^{\dag}$ Laboratoire d’Hydraulique Saint-Venant -- ENPC, CEREMA, EDF R\&D, Chatou, France,\\
    $^*$ Institut Denis Poisson, Universit\'e d'Orl\'eans, Universit\'e de Tours, \\
    CNRS UMR 7013, BP 6759, F-45067 Orl\'eans Cedex 2, France\\{\tt francois.james@univ-orleans.fr}
  }
\end{center}

\begin{abstract} 
  
The derivation of shallow water models from Navier-Stokes equations is revisited yielding a class of two-layer shallow water models.
An improved velocity profile is proposed, based on the superposition of an ideal fluid and a viscous layer inspired by the Interactive Boundary Layer interaction used in aeronautics.  This leads to a new friction law  which
depends not only on velocity and depth but of the variations of velocity and thickness of boundary layer.
The resulting system is an extended shallow water model consisting of three depth-integrated equations: the first two are mass and momentum conservation in which a slight correction on hydrostatic pressure has been made; the third one, known as von K\'arm\'an equation, describes the evolution of the viscous layer. 
  This coupled model is shown to be conditionally hyperbolic, and a Godunov-type finite volume scheme is also proposed.
  Several numerical examples are provided and compared to the ``Multi-Layer Saint-Venant'' model. 
  They emphasize the  ability of the model to deal with unsteady viscous effects. They  illustrate  also the phase-lag between friction and topography, and even recover possible reverse flows.  
  \end{abstract}
  
\small{

}

\pagestyle{plain}
\def\indic{\hbox{1\kern-2,44ptI}}


\section*{Introduction}
\addcontentsline{toc}{section}{\protect\numberline{}Introduction}

Many phenomena in fluvial or maritime hydraulics  involve free surface flows in shallow waters for the study e.g. of floods and tides. Shallow water equations were originally introduced by Saint-Venant in 1871 \cite{saintvenant71} in the context of channel modelling. Since then, the model has been widely extended and is used in the modelling and numerical simulation of a number of natural or man-made phenomena such as river flow \cite{Goutal97,Burguete04}, flood forecasting \cite{Caleffi03,Kirst16}, pollutant transport \cite{Rivlin1995,Govindaraju1996}, dam-break \cite{Alcrudo01,Valiani02}, tsunami \cite{George06,Kim07,Popinet11}, overland flow \cite{Esteves00,Tatard08,Delestre09b}, soil erosion \cite{Castro2008,Le2015} and many others.

The shallow water system can be derived from the incompressible Navier-Stokes equations under several hypotheses; the main one being the {\em long wave approximation} meaning that the characteristic wavelength is much larger than the water depth (see figure \ref{fig:domain} for a sketch and definitions). Two consequences follow then: the hydrostatic pressure law, and the viscous term vanishing in the horizontal direction. Next, to proceed from Navier-Stokes to shallow water, the equations are integrated along the vertical direction. At this point, care has to be taken of the vertical velocity profile, which on the one hand has to be approximated to deal with nonlinearities of the momentum flux, but on the other hand drives the bottom boundary condition, hence the friction phenomena.

Two classical assumptions on the longitudinal velocity profile along the vertical direction lead to explicit integrations. The first one is a viscous Poiseuille-like (i.e. parabolic) profile on the whole water depth which gives rise to a linear (with respect to the depth-averaged velocity) friction term, sometimes referred to as laminar friction. The second one is a constant profile, somehow corresponding to an ideal fluid; but, by construction, there is {\em a priori} no friction term in the integrated equations. 
Friction has to be added afterwards using empirical laws such as Manning, Ch\'ezy, etc (see e.g. \cite{Chow59}). The main drawback of these classical approaches is the non-adaptability of the friction terms for large variations of velocity because the assumed profiles (parabolic or flat) do not hold.

We intend here to pay a particular attention to the fact that these empirical laws are unable to describe the fluid inertia or more precisely to predict the {\em phase-lag} between the bottom friction and a perturbation of the bed which is known as an essential mechanism for dune or ripple formation \cite{Kennedy1963}. Indeed, it is well reported that for the case of a sub-critical flow on a bump, the maximum of the friction must be slightly shifted upstream of the crest to drag the particles from the troughs up to the crest. Consequently, coupling classical  shallow water equations with a mass conservation equation for sediment transport (e.g. Exner equation \cite{Exner1925} for bedload case) cannot predict the bed instability, see \cite{Charru2013,Kouakou2006} for more details.  
We look for a more flexible model with a better understanding of how the no-slip boundary condition gives rise to the friction term in the depth-integrated equations.
This will allow to recover this phase-lag phenomenon as well as boundary layer separation, a manifestation of the recirculation of the flow near bottom.

This is done by using an  asymptotic description of the fluid as a superposition of an ideal fluid over a viscous layer located at the bottom, with a strong interaction between the layers.
The thickness of the viscous layer  is quantified by a small parameter $\bar\delta$ related to the inverse of the Reynolds number of the flow. 
Integrating the incompressibility equation under this consideration leads to the same mass conservation equation of the usual shallow water system. On the contrary, the integration of the momentum equation exhibits major differences. On the one hand,  it turns out that the order of magnitude of the friction term is
larger as expect: 
precisely of order $\bar\delta$, while the above mentioned Poiseuille profile leads to a $\bar\delta^2$ order of magnitude. In the case of an ideal fluid $(\bar\delta=0)$, the model degenerates, of course, into the classical shallow water one.
On the other hand, we introduce a new closure for the momentum flux which involves an additional pressure law of order $\bar\delta$. At this stage, we obtain a system of two equations which are similar in structure with the usual shallow water system, but involving several additional unknown functions related to the viscous layer. 

The next step consists therefore in a careful analysis of the viscous layer. Following a classical methodology in aerodynamics, see e.g. \cite{Schlichting1968}, we integrate the Prandtl equation along the vertical axis to obtain the so-called von K\'arm\'an equation. It describes the evolution of the so-called displacement thickness $\delta_1$ (see Figure \ref{fig:domain}), which is involved in the definition of the afore mentioned corrective pressure, and can be interpreted as some physical thickness of the viscous layer. The system has to be complemented by the velocity equation of the ideal fluid, since it is involved in the von K\'arm\'an equation. We will discuss on the effects for flows over short bumps. The acceleration induced by the bump will change a lot the basic flow so that the shape velocity profile is no longer a half-Poiseuille nor a flat one. This study aims to understand this kind of flows which are not taken into account by the shallow water equations themselves. Furthermore, we will present the link between our system and the  Multi Layer Saint Venant  model proposed in \cite{Audusse2011a}.

The outline of the paper is as follows. In a first section we recall the Navier-Stokes system, and state the long wave approximation which is convenient for shallow water approximation. Next we turn to the viscous layer analysis, and derive Prandtl and von K\'arm\'an equations. Velocity profiles are also introduced. The third section is devoted to the derivation of various formulations of the Extended Shallow Water System. In Section 4 we derive some formal properties of the model, together with the numerical scheme.
 Finally, we evidence several properties of the model by numerical simulations.

\section{From Navier-Stokes to shallow water equations}
In this section we recall how classical models for shallow waters are obtained from Navier-Stokes equations. The first assumption is a long wave approximation, stating that indeed we deal with a thin layer of water. Next, we integrate along the vertical direction, assuming a given velocity profile on the whole water depth.

\subsection{Navier-Stokes equations}
We consider a fluid evolving in a time-dependant domain $\Omega_t=\RR\times \{f_b(x)\le y \le\eta(t,x)\}$.
This thin layer is limited below by a fixed bottom represented by a function $y=f_b(x)$ and above by the free surface described by $y=\eta(t,x)$. We denote the water depth $h=\eta-f_b$. In this study, the properties of the air above the free surface are completely neglected (see Figure \ref{fig:domain}).

\begin{figure}
  \begin{center}
    \includegraphics[width=11cm]{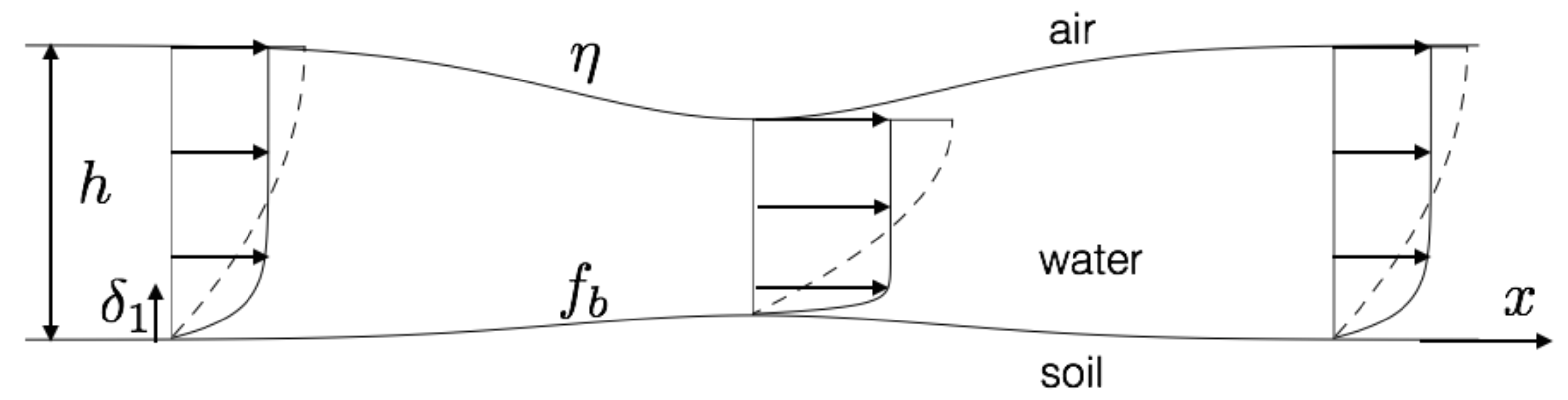}
  \end{center}
  \caption{{\small Domain under consideration: the water layer is defined by the depth $h(x,t)$ of characteristic value $h_0$, the bottom is a given function $f_b$ of characteristic length $L$ and $\eta$ is the free surface.  Two families of velocity profiles are displayed for the flow over the topography, first with the usual half-Poiseuille description (dashed), and second with the flat profile with a boundary layer (plain). Note that the shear (slope of the velocity at the wall) is completely different in those two descriptions even flux and depth are the same.}}\label{fig:domain}
\end{figure}

Our starting point is to consider the dimensionless Navier-Stokes equations expressing the mass and momentum conservations of an incompressible Newtonian fluid \cite{Schlichting1968}. For the sake of simplicity, we limit ourselves in this work to consider only laminar flows. Indeed, the asymptotic from the Navier-Stokes equations is clearer and the resulting description is quantitative. A similar study for turbulent flows can be made with a modified Reynolds tensor. Nondimensionalization has been made with the same characteristic length $h_0$, e.g. a reference water depth, for both the abscissa and the ordinate. The dimensionless Navier-Stokes equations write
\begin{align}
  \label{Mass}
  \partial_{x}u+\partial_{y}v &=0,\\
  \label{Momentum1}
  \partial_{t}u+u\partial_{x}u+v\partial_{y}u &=-\partial_{x}p+\frac{1}{Re}\nabla^2 u,\\
  \label{Momentum2}
  \partial_{t}v+u\partial_{x}v+v\partial_{y}v &=-\partial_{y}p-\frac{1}{Fr^2}+\frac{1}{Re}\nabla^2 v,
\end{align}
where $u,v$ are the horizontal and vertical velocities respectively and $p$ is the pressure. We have defined the Reynolds $Re$ and Froude $Fr$ numbers given by 
\begin{equation*}
  Re := \frac{u_0h_0}{\nu}, \quad Fr := \frac{u_0}{\sqrt{gh_0}}
\end{equation*}
in which $u_0$ and $\nu$ being the reference velocity and kinematic viscosity respectively and the constant $g$ stands for the acceleration due to gravity. The Reynolds number expresses the ratio between the inertia force and the viscosity; the Froude number represents the ratio between the kinetic and potential energies. 

The system is complemented with the following boundary conditions:
\begin{itemize}
\item at the bottom $y=f_b(x)$: no-slip condition, i.e. $u=v=0$,
\item at the free surface $y=\eta(t,x)$:
  \begin{itemize}
  \item kinematic boundary condition: $v=\partial_t\eta+u\partial_x\eta$
  \item continuity of the stress tensor: $\sigma\cdot{\bf n}=0$, where 
    ${\bf\sigma}=\begin{pmatrix}
    2\partial_xu-p&\partial_xv+\partial_yu\\
    \partial_yu+\partial_xv&2\partial_yv-p
  \end{pmatrix}$ is the stress tensor and ${\bf n}=\dfrac1{\sqrt{1+(\partial_x\eta)^2}}\begin{pmatrix}\partial_x\eta\\-1\end{pmatrix}$ is the outer unit normal to the free surface.
  \end{itemize}
\end{itemize}

\subsection{Long wave scaling}
Up to now, no hypothesis has been taken into account for the size order of the characteristic quantities $u_0,h_0$. We have in mind applications to rivers or coastal flows for which the following conditions may be observed: 
\begin{itemize}
\item the Reynolds number is large,
\item the horizontal velocity has small variation along the vertical,
\item the vertical velocity is small compared to the horizontal velocity.
\end{itemize}
We introduce the aspect ratio 
\begin{equation*}
  \epsilon := \frac{h_0}{L} \ll 1,
\end{equation*}
where $L$ is a characteristic wave length and $h_0$ a characteristic depth. Let us start by investigating the third one, which justifies the following scaling for the velocities:
\begin{equation*}
  v=\epsilon \tilde v, \quad u=\tilde u.
\end{equation*}
Then the mass conservation equation (\ref{Mass}) enforces also a scaling for the space variables $y \ll x$ since
\begin{equation*}
  0=\partial_{x}u+\partial_{y}v =\partial_x\tilde u+\epsilon \partial_y\tilde v.
\end{equation*}
Hence there are two options for the variable scaling:
\begin{enumerate}
\item Long wave scaling:
  \begin{equation*}
    x=\frac{\tilde x}{\epsilon}, \quad y=\tilde y, \quad t=\frac{\tilde t}{\epsilon}.
  \end{equation*}
  It means that the long wave hypothesis needs a long time study. Furthermore, since $f_b(x)=\tilde{f_b}(\tilde x)$ we have $f_b'(x)=\epsilon \tilde{f_b}'(\tilde x)$ and so the bottom slope needs to be small enough.
\item Thin layer scaling:
  \begin{equation*}
    x=\tilde x, \quad y=\epsilon \tilde y, \quad t=\tilde t
  \end{equation*}
  This scaling restricts the study of small vertical velocity only to a thin water depth which tends to zero when $\epsilon\to 0$. It is the classical scaling used in the boundary layer approach.
\end{enumerate}
So, the long wave scaling is compatible with the hypothesis about small vertical velocity compared to horizontal velocity. Let us see the consequences for the set of equations (\ref{Mass})--(\ref{Momentum2}) and the boundary conditions:
\begin{align}
  \partial_{\tilde x}\tilde u+\partial_{\tilde y}\tilde v & = 0, \nonumber \\
  \label{tildeM1}
  \epsilon\left[\partial_{\tilde t}\tilde u +\tilde u\partial_{\tilde x}\tilde u+ \tilde v\partial_{\tilde y} \tilde u\right] & = -\epsilon \partial_{\tilde x}\tilde p+\frac{1}{Re}\left[\epsilon^2\partial_{\tilde x}^2\tilde u+\partial_{\tilde y}^2\tilde u\right],\\
  \label{tildeM2}
  \epsilon^2\left[\partial_{\tilde t}\tilde v+\tilde u\partial_{\tilde x}\tilde v+\tilde v \partial_{\tilde y}\tilde v\right] & = -\frac{1}{Fr^2}-\partial_{\tilde y}\tilde p+\frac{\epsilon}{Re}\left[\epsilon^2 \partial_{\tilde x}^2\tilde v+\partial_{\tilde y}^2\tilde v\right],\\
  \tilde u=\tilde v & = 0 \quad \textrm{at}\quad \tilde y=\tilde{f_b},\nonumber \\
  \partial_{\tilde t}\tilde \eta+\tilde u\partial_{\tilde x}\tilde \eta - \tilde v & = 0 \quad \textrm{at}\quad \tilde y=\tilde \eta,\nonumber  \\
  \begin{pmatrix}
    \epsilon \big((2\epsilon \partial_{\tilde x}\tilde u-\tilde p)\partial_{\tilde x}\tilde \eta-\epsilon \partial_{\tilde x}\tilde v\big)-\partial_{\tilde y}\tilde u\\
    \epsilon \big(\partial_{\tilde x}\tilde \eta(\partial_{\tilde y }\tilde u+\epsilon \partial_{\tilde x}\tilde v)-2\partial_{\tilde y}\tilde v\big)-\tilde p
  \end{pmatrix} &=  \begin{pmatrix}0\\0\end{pmatrix} \quad \textrm{at} \quad \tilde y=\tilde \eta. \nonumber
\end{align}
Taking an approximation at order $O(\epsilon)$ leads to:
\begin{itemize}
\item Cancellation of the viscosity in horizontal direction. Equation (\ref{tildeM1}) reduces to:
  \begin{equation*}
    \partial_{\tilde t}\tilde u +\tilde u\partial_{\tilde x}\tilde u+ \tilde v\partial_{\tilde y} \tilde u = -\partial_{\tilde x}\tilde p+\frac{1}{\epsilon Re}\partial_{\tilde y}^2\tilde u.
  \end{equation*}
\item Simplified version of the stress tensor continuity at the free surface: 
  \begin{equation*}
    \tilde p=0, \quad \partial_{\tilde y}\tilde u=0\quad \textrm{at} \quad \tilde y=\tilde\eta.
  \end{equation*}
\item Hydrostatic pressure with (\ref{tildeM2}) and $\tilde p=0$ at the surface:
  \begin{equation*}
    \partial_{\tilde y}\tilde p=-\frac{1}{Fr^2} \Longrightarrow \tilde p=\frac{1}{Fr^2}(\tilde \eta-\tilde y).
  \end{equation*}
  This approximation implies that $\partial_{\tilde x}\tilde p$ does not depend on $\bar y$. In other words, the pressure gradient is conserved over the vertical. This result for pressure at order $O(\epsilon)$ is already observed (see e.g. \cite{Kalliadasis2011}).
\end{itemize}
\begin{remark}
  We emphasize here that the above properties, which are classical shallow water hypotheses, are brought out solely by the long wave approximation.
\end{remark}

To summarize, the long wave approximation of the Navier-Stokes equations consists of the following system, sometimes called the RNSP equations (Reduced Navier-Stokes/Prandtl \cite{Lagree2005}). This will be our reference system for the remaining of this article. In these equations, it is more convenient to define an effective Reynolds $Re_h$ number which takes into account the aspect ratio of the model
\begin{equation*}
  Re_h := \epsilon Re.
\end{equation*}
Dropping all the tildes from the variables and unknowns, the system can be written:
\begin{align}
  \label{LWmass}
  \partial_x u + \partial_y v &=0,\\
  \label{LWmomentum}
  \partial_t u + u\partial_x u + v\partial_y u & =-\partial_x p + \frac{1}{Re_h}\partial_y^2 u,\\
  \label{Hydrop}
  \partial_y p &= -\frac{1}{Fr^2},\\
  \label{kinematic}
  \partial_t \eta + u\partial_x \eta - v &= 0\quad \textrm{at}\quad y = \eta,\\
  \label{ApproxStress}
  p = 0, \quad \partial_y u & =0 \quad \textrm{at}\quad y = \eta, \\
  \label{eq:no-slip}
  u = v &=0 \quad \textrm{at} \quad y = f_b.
\end{align}

In the limit $Re_h\to\infty$, the case of an incompressible ideal fluid, the RNSP equations degenerate to the hydrostatic Euler system.
Indeed, the equations of mass conservation \eqref{LWmass}, hydrostatic pressure \eqref{Hydrop} and boundary condition at free surface \eqref{kinematic}-\eqref{ApproxStress} are unchanged, but the momentum conservation equation is replaced by
\begin{equation}\label{eq:Emomentum}
  \partial_t u + u\partial_x u + v\partial_y u = -\partial_x p.
\end{equation}
We conclude this paragraph by giving the behaviour of horizontal velocity.
\begin{prop}\label{eqUe}
  The horizontal velocity $u$ of a smooth solution of the hydrostatic Euler system is constant along the vertical direction, so we write $u := u_e(t,x)$, and verifies the following equation:
  \begin{equation}\label{eq:FP}
    \partial_tu_e+u_e\partial_xu_e=-\partial_x p.
  \end{equation}
\end{prop}

\begin{proof}
  Applying a partial derivative in $y$ on equation \eqref{eq:Emomentum} and taking into account the incompressibility \eqref{LWmass} together with the hydrostatic pressure law \eqref{Hydrop} leads to $\frac{\rm d}{{\rm d}t} (\partial_y u) = 0$. It means that $\partial_y u$ remains constant along the characteristic curve $x'(t) = u, ~ y'(t) = v$, hence $\partial_y u = \partial_y u|_{(x_0,\eta_0)} = 0$ by condition \eqref{ApproxStress} where $(x_0, \eta_0)$ is the foot of characteristic curve starting at the free surface $y = \eta$. Consequently the horizontal velocity $u$ of hydrostatic Euler system is independent on $y$.
\end{proof}

As it is well-known in Ideal Fluid theory, the depth-independent horizontal velocity component leads to a slip velocity at the bed. Thus, the no-slip boundary condition \eqref{eq:no-slip} is not relevant for ideal fluid and has to be replaced by a weaker one called the {\em non-penetration} condition
\begin{equation}\label{eq:non-penetration}
  v_e|_{y = f_b}  =  u_e {f_b}'.
\end{equation}

Recovering some connection between the ideal fluid equations and the no-slip condition is precisely the aim of the viscous layer theory, which we will present in section \ref{sec:BoundLayer}.

\subsection{Classical shallow water equations}\label{subsec:SW}
Let us recall briefly the classical way to obtain the shallow water model by vertical integration the RNSP equations over the whole water depth. 

\begin{defn}\label{def:h-U}
The water depth $h$ and the depth-averaged horizontal velocity $U$ are given by
\begin{equation}\label{eq:definition-h-U}
  h := \eta - f_b, \quad hU := \int_{f_b}^{\eta}u\d{y}.
\end{equation}
\end{defn}

Integrating the mass conservation equation \eqref{LWmass}, from $f_b$ to $\eta$, yields
\begin{align*}
  0 &= \int_{f_b}^{\eta}\partial_x u\d{y} + v|_{y = \eta} - v|_{y = f_b} \\
  &= \partial_x \left(\int_{f_b}^{\eta}u\d{y}\right) - u|_{y = \eta}\partial_x\eta + u|_{y = f_b'}f_b' + v|_{y = \eta} - v|_{y = f_b} \\
  &= \partial_x \left(\int_{f_b}^{\eta}u\d{y}\right) + \partial_t\eta 
\end{align*}
in which we have applied the kinematic condition \eqref{kinematic} at the surface and the no-slip condition \eqref{eq:no-slip}---or the weaker one \eqref{eq:non-penetration}---at the bottom. As $f_b$ is time-independent, we can write $\partial_t \eta = \partial_t (\eta-f_b)$; and by using definition \eqref{eq:definition-h-U}, the mass conservation in its integrated form reads
\begin{equation}\label{eq:sw-mass}
  \partial_t h+\partial_x(hU)=0.
\end{equation}

In the same way, we now derive the depth-integrated momentum balance equation. The main point to notice at this stage is the appearance of the nonlinear momentum flux and the friction term. Indeed, by integrating the momentum equation \eqref{LWmomentum} over the depth and using the fact that $\partial_xp$ does not depend on $y$, and also the condition $\partial_y u|_{y = \eta} = 0$ from \eqref{ApproxStress}, we get
\begin{align*}
  -\frac{1}{Re_h}\partial_y u|_{y = f_b} - h\partial_y p & =  \int_{f_b}^{\eta}\partial_t u \d{y} +\int_{f_b}^{\eta} \left(u \partial_x u + v\partial_y u\right) \d{y} \\
  & = \int_{f_b}^{\eta}\partial_t u \d{y} + \int_{f_b}^{\eta}2u\partial_xu\d{y} + uv|_{y =\eta} - uv|_{y = f_b'} \\
  & = \partial_t\left(\int_{f_b}^{\eta}u\d{y}\right) + \partial_x\left(\int_{f_b}^{\eta}u^2\d{y}\right) -u(\partial_t\eta + u\partial_x \eta - v)|_{y = \eta} + u(uf_b' -v)|_{y = f_b'}.
\end{align*}

Using definition \eqref{eq:definition-h-U} and applying the free surface condition \eqref{kinematic} and the bottom condition \eqref{eq:no-slip} or \eqref{eq:non-penetration}, we can rewrite the integrated momentum conservation equation in the form
\begin{equation}\label{eq:sw-momentum}
  \partial_t(hU)+\partial_x (\beta hU^2) = -h\partial_x p - \tau_b,
\end{equation}
in which we have introduced the so-called Boussinesq coefficient $\beta$ and the bottom shear stress $\tau_b$, also called {\em friction}, which are defined by 
\begin{equation}\label{eq:def-flux-friction}
  \int_{f_b}^{\eta}u^2\, dy := \beta hU^2, \quad \tau_b := \frac{1}{Re_h}\partial_y u|_{y = f_b}.
\end{equation}
Therefore, evaluating $\beta$ and $\tau_b$ requires the knowledge of the flow. It can be checked that $\beta \geq 1$ since, by definition \eqref{eq:def-flux-friction}, we can write
\begin{equation}\label{eq:Boussinesq}
  \beta = 1+\frac{1}{h}\int_{f_b}^{\eta}\left(1-\frac{u}{U}\right)^2\d{y}.
\end{equation}
Without complementary equations, a closure relation on the velocity profile is needed in order to compute the Boussinesq coefficient and to express the friction term in function of the conservative variables $(h, hU)$. Let us recall two constitutive profiles which are often adopted in the context of shallow water flows. 

\medskip
\noindent{\bf Flat profile.} This is the most classical approach in hydraulic river modelling. Scaling analysis reveals that the velocity profile is quasi-flat except within a very thin-layer close to the river bed. Based on this consideration, 
the velocity profile can be assumed to be flat over the whole water depth, so that $\beta$ equals one.
The friction term $\tau_b$ is not properly defined in this context.
The flow can be described using the Euler system resulting in an inviscid shallow water model -- equation \eqref{eq:sw-momentum} without the friction term $\tau_b$. 

 A large family of empirical friction laws exist, which express the friction as a quadratic function of $U$ with a friction coefficient $C_f = O(Re^{-1/4})$ for a smooth bottom, see \cite{Schlichting1968}. This coefficient depends on $h$ and $U$ as well, for instance with Chézy, Manning laws, see \cite{Chow59} for a bibliographical study. In summary, this kind of model consists in writing 
\begin{equation}\label{eq:flat-profile}
  \beta = 1, \quad \tau_b = \frac{1}{2}C_f U^2.
\end{equation}
We emphasize again that the friction term derives here from empirical considerations.

\medskip
\noindent{\bf Poiseuille profile.} This type of profile is inspired from a analytic solution of the RNSP equations in the case of an uniform flow on a negative constant slope. The balance between the friction and the driving force of the slope gives a {\em self-similar} parabolic solution, also known as the half-Poiseuille or Nusselt solution: 
\begin{equation*}
  \frac{u}{U} = 3\left(\zeta - \frac{1}{2}\zeta^2\right), \quad 0 \leq \zeta := \frac{y-f_b}{h} \leq 1.
\end{equation*}
This choice of profile leads to
\begin{equation}\label{eq:Poiseuille-profile}
  \beta = \frac{6}{5}, \quad \tau_b = \frac{3}{Re_h}\frac{U}{h}.
\end{equation}
The friction is indeed linear with respect to the mean velocity, and is referred to as laminar friction. 

Using such a prescribed profile in shallow water equations leads to some important restriction of the model, in particular when dealing with large variation of the velocity. For example it has been reported in \cite{Hogg2004} that a constant value for the Boussinesq coefficient is less adapted to describe the dynamics of the fluid layer close to a dry-wet transition.

We conclude this section by evidencing that  both these models do not present a phase-lag for the flow over a bump.
Let us  study a steady  linearized solution of the usual shallow water equations. Consider such a small perturbation of the bed $f_b$ that we can write in the form
$ f_b = \epsilon f_b^1,$where $\epsilon$ is just a small parameter and not necessarily the aspect ratio defined before. We look for the solution in the form
\begin{equation}\label{eq:lin-form}
  h = h^0 + \epsilon h^1, \quad U = U^0 + \epsilon U^1.
\end{equation}
For high Reynolds numbers, we see from relations \eqref{eq:flat-profile} and \eqref{eq:Poiseuille-profile} that the friction is negligible. We can consider therefore $h,U$ solution of the following frictionless and steady state shallow water equations
\begin{equation}\label{eq:steady-sw}
  \partial_x(hU) = 0, \quad \partial_x\left(\beta hU^2 + \frac{ 1}{2Fr^2}h^2\right) = -\frac{1}{Fr^2}hf_b'.
\end{equation}
Inserting \eqref{eq:lin-form} in \eqref{eq:steady-sw} and identifying powers of $\epsilon$ leads to a cascade of equations for each terms $h^i, U^i, ~i = 0,1$. Then, it should be checked that the zero-order terms $h^0$ and $U^0$ are needed constant. For first-order terms, a straightforward calculation leads to  
\begin{equation*}
  h^0\partial_x U^1 + U^0\partial_x h^1 = 0, \quad \left(\frac{h^0}{Fr^2} - \beta(U^0)^2 \right)\partial_x h^1 = - \frac{h^0}{Fr^2} (f_b^1)'.
\end{equation*}
By introducing the local Froude number $Fr_0$, we can express the linearized solution in the form
\begin{equation}\label{eq:lin-sol-sw}
  Fr_0^2 := \frac{\beta(U^0)^2}{h^0}Fr^2, \quad h = h^0 + \frac{1}{Fr_0^2 - 1}f_b, \quad U = U^0 + \frac{U^0}{h^0}\frac{1}{1-Fr_0^2}f_b.
\end{equation}
As we can see, $U$ is exactly {\em in-phase} with $f_b$. As a consequence, local maxima of the friction estimated by empirical formulas \eqref{eq:flat-profile} or \eqref{eq:Poiseuille-profile} are always reached at $f_b'=0$, that is at the crest of the bump. Indeed, because $\partial_x h = \partial_x U =0$ if $f_b' = 0$ by \eqref{eq:lin-sol-sw}, it follows that $\partial_x \tau_b = \partial_h\tau_b \partial_x h + \partial_U\tau_b \partial_x U = 0$.

\section{Viscous layer analysis}\label{sec:BoundLayer}
We turn now to the main step towards the model we look for. It mainly consists in dividing the fluid in two layers:
\begin{itemize}
\item an ideal fluid layer dealing with the free surface;
\item a thin viscous layer with the no-slip condition at the bottom.
\end{itemize}
In the first layer addressing to ideal fluid, we take advantage of the explicit integration along the vertical. In the second one describing viscous layer, we take into account the viscosity in the vertical direction and recover some friction in the integrated equations. This section is devoted to the study of the viscous layer, and to the analysis of the interactions between the two layers. 

We introduce a small parameter $\bar \delta$, whose magnitude will be specified below. It is related to the thickness of the viscous layer, but does not correspond to its actual physical value. We follow the classical strategy used in the boundary layer theory \cite{PRANDTL1928,Schlichting1968} except that in that case $\bar \delta\to 0$, whereas we keep a finite value here. The first step is to rescale again the RNSP equations with the thin layer scaling to obtain a set of the well-known Prandtl equations. The next step consists in vertical integration of these equations over the viscous layer. This leads to the so-called von K\'arm\'an equation, where extra unknowns are introduced. Finally, some suitable assumptions have to be made on the velocity profile in order to obtain a closed model.

\subsection{Prandtl equations}
We introduce the following change of variables, referred to as the Prandtl shift:
\begin{equation}\label{PrandtlShift}
  x=\bar x,\quad y=\bar \delta \bar y +f_b, \quad t=\bar t,\quad  p=\bar p,\quad
  \bar u= u,\quad \bar v=\frac{ v-f_b' u}{\bar \delta}.
\end{equation}	
By this, the RNSP equations \eqref{LWmass}, \eqref{LWmomentum}, \eqref{Hydrop} and \eqref{eq:no-slip} are transformed  into a set of boundary layer equations on a {\em flat bottom}
\begin{align*}
  \partial_{\bar x}\bar u+\partial_{\bar y}\bar v & = 0,\\
  \partial_{\bar t}\bar u+\bar u\partial_{\bar x} \bar u+\bar v\partial_{\bar y}\bar u &= -\partial_{\bar x}\bar p+\frac{f_b'}{\bar \delta}\partial_{\bar y}\bar p+\frac{1}{Re_h \bar \delta^2}\partial_{\bar y}^2\bar u,\\
  \frac{1}{\bar \delta}\partial_{\bar y}\bar p & = -\frac{1}{Fr^2},\\
  \bar u=\bar v &= 0 \quad \textrm{at} \quad \bar y=0.
\end{align*}

There are several possible scalings for $\bar\delta$ in terms of $Re_h$:
\begin{itemize}
\item if $\bar \delta$ verifies $Re_h\bar \delta^2 \gg 1$, we recover the ideal fluid equations;
\item if $\bar \delta$ satisfies $ Re_h\bar \delta^2\ll 1$, we obtain 
  $\partial_{\bar y}^2\bar u=0$ that leads to $\bar u=0$ due to the continuity of the stress tensor and the no-slip condition. So we do not consider this trivial case;
\item the last possibility is $Re_h\bar \delta^2\sim 1$, which balances the convective terms and the diffusive one. It is called {\em dominant balance} or  {\em least degeneracy principle} \cite{VanDyke1975}, 
and allows to preserve as many as possible terms in the equations.
\end{itemize}
This is why in what follows, we consider the scaling
\begin{equation}\label{Scaling}
  \bar \delta=\frac{1}{\sqrt{Re_h}} \ll 1.
\end{equation}
With this choice of $\bar\delta$, the viscous term appears with the same order as the other terms in the momentum equation. We obtain the Prandtl equations written in viscous layer variables:
\begin{numcases}{}
  \partial_{\bar x}\bar u+\partial_{\bar y} \bar v=0&\label{Eccl1}\\
  \partial_{\bar t} \bar u+\bar u \partial_{\bar x} \bar u +\bar v \partial_{\bar y} \bar u  = -\partial_{\bar x} \bar p - \dfrac{f'_b}{Fr^2} + \partial_{\bar y}^2 \bar u &\label{Eccl2}\\ 
  \partial_{\bar y}\bar p=-\dfrac{\bar\delta}{Fr^2}&\label{Eccl3}\\
  \bar u=\bar v=0\quad \textrm{ when }\; \bar y=0& \label{Eccl4}
\end{numcases}
We notice that this system of equations is in the same form as the Prandtl equations obtained directly from Navier-Stokes equations with classical boundary layer scaling (\cite {Schlichting1968}, ch. VII) except for the topography term in the momentum equation \eqref{Eccl2}.

Up to now, we do not have enough boundary conditions for the viscous layer. The natural connection consists in assuming that the velocity at the ``top'' of the viscous layer coincides with the velocity of inviscid layer. Precisely, we impose the following matching boundary condition
\begin{equation}\label{Eccl5}
  \bar u(\bar t,\bar x, \bar\eta ) = u_e(t,x), \quad\mbox{where } \bar\eta := \frac{\eta - f_b}{\bar\delta},
\end{equation}
which is obviously compatible with the Prandtl shift \eqref{PrandtlShift} since $\bar x=x$ and $\bar t=t$. Notice that in classical boundary layer theory the limit is given by the {\em asymptotic} matching $\bar u(\bar t, \bar x,\bar y \rightarrow \infty ) \rightarrow u_e(t,x)$.

\subsection{Von K\'arm\'an equation}
The von K\'arm\'an equation expresses the defect of velocity between the ideal fluid and the viscous layer. A classical way to obtain such an equation consists in writing the Prandtl equation, then introducing the  velocity defect $(u_e-\bar u)$, and finally integrating it on the viscous layer (see Schlichting \cite{Schlichting1968}). We introduce the following two integrated quantities, see Figure \ref{Fig:definition-delta1}:
\begin{defn}\label{def:delta1}  Let $U$ be the depth-averaged velocity. We define
  \begin{itemize}
  \item the displacement thickness $\delta_1$ given by 
    \begin{equation}\label{tildeUfctue}
      hU = (h-\bar\delta\delta_1)u_e,
    \end{equation}
  \item the momentum thickness $\delta_2$ given by 
    \begin{equation}\label{usquarefctue}
      \int_{f_b}^{\eta}u^2\d{y} = \big(h-\bar\delta(\delta_1+\delta_2)\big)u_e^2.
    \end{equation}
  \end{itemize}
\end{defn}
Physically, the displacement thickness expresses the distance by which the ground should be displaced to obtain an ideal fluid with velocity $u_e$ and the same flow rate $hU$ (see Figure \ref{Fig:definition-delta1}). In the same way, the momentum thickness accounts for the loss of momentum in the viscous layer.

\begin{figure}[ht!]\centering
  \begin{tikzpicture}[domain=-1:8, scale=0.8]
    \draw[-] (0,0) -- (8,0);
    \fill[pattern=north east lines] (0,0) -- (8,0) -- (8,-0.2) -- (0,-0.2);
    \draw[->] (4,0.5) -- (5,0.5);
    \draw[->] (4,1) -- (5,1.);
    \draw[->] (4,1.5) -- (5,1.5) node[right]{$u_e$};
    \draw[->] (4,2.) -- (5,2);
    \draw[->] (4,2.5) -- (5,2.5);
    \draw[black,dashed] (4,0.5) -- (4,2.5);
    \draw[black,dashed] (5,0.5) -- (5,2.5);
    \draw[thin, dashed, red] (1,0) -- (1,2.5);
    \draw[->, red] (1,0.1) -- (1.15,0.1);
    \draw[->, red] (1,0.5) -- (1.6,0.5);
    \draw[->, red] (1,1) -- (1.9,1);
    \draw[->, red] (1,1.5) -- (2,1.5)  node[right]{$u$};
    \draw[->, red] (1,2) -- (2,2);
    \draw[->, red] (1,2.5) -- (2,2.5);
    \draw[red] plot[domain=0:1.5] ({1+(1+(\x/1.5+1)*(\x/1.5-1)^3)}, {(\x)})  ;
    \draw[->, red] (2,1.5) -- (2,2.5);
    \draw[black,dashed] (0.5,0.5) -- (6,0.5);
    \draw[blue,<->] (6.,0) -- (6.,0.5) node[right]{$\bar\delta\delta_1$};
    \draw[blue,<->] (7.,0) -- (7.,2.5) node[right]{$h$};
  \end{tikzpicture}
  \caption{Interpretation of the displacement thickness, the flux of mass is the same in the viscous layer and in a equivalent layer of ideal fluid shifted by an amount of $\bar\delta\delta_1$.}
  \label{Fig:definition-delta1} 
\end{figure}
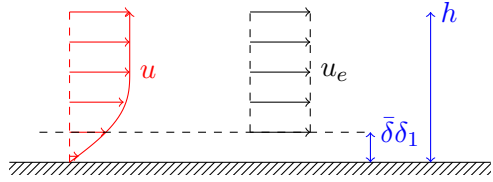
A simple computation from \eqref{tildeUfctue} and \eqref{usquarefctue} leads to the following expressions for these quantities:
\begin{equation*}
  \delta_1 = \int_0^{\bar\eta}\left(1-\dfrac{\bar u}{u_e}\right)\,\d{\bar y}, \quad 
  \delta_2 = \int_0^{\bar\eta}\dfrac{\bar u}{u_e}\left(1-\dfrac{\bar u}{u_e}\right)\,\d{\bar y}.
\end{equation*}
In the limit $\bar\delta\to 0$, we recover the classical formul\ae\ for $\delta_1, ~\delta_2$ in the boundary layer scaling \cite{Schlichting1968}:
\begin{equation*}
  \delta_1 = \int_0^{+\infty}\left(1-\dfrac{\bar u}{u_e}\right)\,\d{\bar y}, \quad 
  \delta_2 = \int_0^{+\infty}\dfrac{\bar u}{u_e}\left(1-\dfrac{\bar u}{u_e}\right)\,\d{\bar y}.
\end{equation*}

\begin{prop}\label{VKcond}
  The evolution of the displacement and momentum thicknesses are ruled by the so-called von K\'arm\'an equation:
  \begin{equation}\label{Equedelta}
    \partial_t(u_e\delta_1) + u_e\delta_1\partial_xu_e + \partial_x(u_e^2\delta_2) = \bar\tau_b,
  \end{equation}
  where $\bar\tau_b$ denotes the parietal constraints:
  \begin{equation}\label{parietal}
    \bar\tau_b := \partial_{\bar y}\bar u|_{\bar y=0} = \frac{\tau_b}{\bar\delta}.
  \end{equation}
\end{prop}

\begin{proof}
  First we notice that the Prandtl shift leads to following relations
  \begin{equation*}
    \partial_{x}=\partial_{\bar x}-\frac{f_b'}{\bar \delta}\partial_{\bar y}, \quad
    \partial_{y}=\dfrac1{\bar\delta}\partial_{\bar y}.
  \end{equation*}
  Momentum equation \eqref{eq:FP} for inviscid flow can be rewritten as
  \begin{equation*}
    \partial_{\bar t} \bar u_e + \bar u_e\partial_{\bar x}\bar u_e 
    = -\partial_{\bar x}\bar p + \dfrac{f'_b}{\bar\delta}\partial_{\bar y}\bar p
    = -\partial_{\bar x}\bar p-\dfrac{f'_b}{Fr^2},
  \end{equation*}
  in which we have used \eqref{Eccl3} to rewrite the right-hand side. The difference between this equation and (\ref{Eccl2}) gives
  \begin{equation*}
    \partial_{\bar t}(\bar u_e-\bar u) + \bar u_e\partial_{\bar x}\bar u_e - \bar u \partial_{\bar x}\bar u 
    - \bar v\partial_{\bar y}\bar u = -\partial_{\bar y}^2\bar u.
  \end{equation*}
  Through (\ref{Eccl1}) and (\ref{Eccl4}), we can rearrange the term $\bar v=-\int_{0}^{\bar y}\partial_{\bar x}\bar u$. Furthermore we get
  \begin{equation*}
    \partial_{\bar t}(\bar u_e-\bar u)+(\bar u_e-\bar u)\partial_{\bar x}\bar u_e
    + \bar u \partial_{\bar x}(\bar u_e-\bar u)
    + \partial_{\bar y}\bar u\int_{0}^{\bar y}\partial_{\bar x}\bar u\,\d{\bar y}
    =-\partial_{\bar y}^2\bar u.
  \end{equation*}
  Using integration by parts, the last term in the left-hand side can be rewritten as
  \begin{equation*}
    \partial_{\bar y}\bar u\int_{0}^{\bar y}\partial_{\bar x}\bar u\,\d{\bar y}
    = -\bar u\partial_{\bar x}\bar u + 
    \partial_{\bar y}\left(\bar u\int_{0}^{\bar y}\partial_{\bar x}\bar u\,\d{\bar y}\right).
  \end{equation*}
  Now we integrate the resulting equation over $\bar y$, between $0$ and $\bar\eta$, together with the matching boundary condition \eqref{Eccl5} to obtain the momentum integral equation
  \begin{align*}
    \partial_{\bar t}\int_{0}^{\bar\eta}(\bar u_e-\bar u)\,\d{\bar y}     + \partial_{\bar x} \bar u_e\int_{0}^{\bar\eta}(\bar u_e-\bar u)\,\d{\bar y} + \int_{0}^{\bar\eta}\bar u\partial_{\bar x}(\bar u_e-\bar u)\,\d{\bar y} & -\int_{0}^{\bar\eta}\bar u\partial_{\bar x}\bar u \,\d{\bar y} + u_e\int_{0}^{\bar\eta}\partial_{\bar x}\bar u\,\d{\bar y} \\
    & = \partial_{\bar y}\bar u|_{\bar y=0}.
  \end{align*}
  The last three terms of the left-hand side are now rewritten as 
  $\partial_{\bar x}\int_{0}^{\bar\eta}\bar u(\bar u_e-\bar u)\,\d{\bar y}$. Moreover, since $u_e$ is independent of $\bar y$, we have the relations
  \begin{equation*}
    u_e\delta_1=\int_{0}^{(\eta-f_b)/\delta} (u_e-\bar u)\,\d{\bar y}, \quad 
    u_e^2\delta_2=\int_{0}^{(\eta-f_b)/\delta}\bar u(u_e-\bar u)\,\d{\bar y}.
  \end{equation*}
  Since $t=\bar t$ and $x = \bar x$ in the Prandtl shift, and all unknowns in the equation are independent on $\bar y$, we can drop the bar symbols in the derivatives. From definition \eqref{parietal} we obtain the final form \eqref{Equedelta} of the von K\'arm\'an equation.   
  Finally, the relation between the friction $\tau_b$ and the rescaled one $\bar\tau_b$ in \eqref{parietal} is an easy consequence of the Prandtl shift and definition \eqref{eq:def-flux-friction}.
\end{proof}

Coupled with equation \eqref{eq:FP} on the velocity $u_e$ of ideal fluid, the von K\'arm\'an equation \eqref{Equedelta} 
 gives only a partial representation of the boundary layer, since it involves four additional unknowns, namely $u_e$, $\delta_1$,
$\delta_2$, and $\bar\tau_b$. To proceed further towards an integrated model, we need to specify velocity profiles to close the von K\'arm\'an equation.

\subsection{Velocity profile in the viscous layer}\label{sec:ViscProf}
Through the viscous layer, the velocity $\bar u$ varies from $0$ (at the bottom) to the  ideal fluid velocity $u_e$. Therefore we introduce a profile function $\varphi$ as well as 
a scaling factor $\Delta(\bar t, \bar x)$, chosen in such a way that $\Delta$ quantifies the physical thickness of the viscous layer. Following \cite{Schlichting1968}, we wish to have
\begin{equation}\label{eq:visous-profile-function}
  \frac{\bar u(\bar t,\bar x,\bar y)}{u_e} = \varphi\left(\frac{\bar y}{\Delta}\right) = \varphi(\xi), \quad \xi := \frac{\bar y}{\Delta(\bar t,\bar x)}.
\end{equation}
Therefore we choose $0 < \Delta\le \bar\eta$, and a profile function $\phi(\xi)$ such that
\begin{equation}\label{profileProp}
  \varphi(0) = 0,\quad
  \varphi(\xi \geq 1) = 1,
  \quad \int_0^1(1-\phi)\d{\xi}:= \alpha_1<+\infty,
  \quad \int_0^1\phi (1-\phi)\d{\xi}:= \alpha_2<+\infty.
\end{equation}
Hence, by definition of $\delta_1$ and $\delta_2$, we can write
\begin{align*}
  \delta_1 & = \int_{0}^{\bar\eta}\left(1-\frac{\bar u}{u_e}\right)\,\d{\bar y}
  =\Delta\int_0^1(1-\phi)\d{\xi} =\Delta \alpha_1, \\
  \delta_2 & = \int_{0}^{\bar\eta}\frac{\bar u}{u_e}\left(1-\frac{\bar u}{u_e}\right)\,\d{\bar y} = \Delta\int_0^1 \varphi(1-\varphi)\d{\xi} 
  =\Delta \alpha_2.
\end{align*}

To link these variables, we introduce the {\em shape factor} $H$ which only depends on the profile function $\varphi$
\begin{equation}\label{delta2fnctdelta1}
  H := \frac{\delta_1}{\delta_2} = \frac{\int_0^1(1-\phi)\d{\xi}}{\int_0^1\varphi(1-\varphi)\d{\xi}} \geq 1.
\end{equation}
The parietal constraints can also be expressed in terms of $\varphi$ and $u_e$ as
\begin{equation}\label{parietalwithprofile}
  \bar\tau_b=\partial_{\bar y}\bar u|_{\bar y=0} = \frac{\varphi'(0)}{\Delta}u_e = \dfrac{\alpha_1\varphi'(0)}{\delta_1}u_e = \dfrac{f_2 H}{\delta_1}u_e,
\end{equation}
where the parameter $f_2$ is known as the {\em friction factor} (see \cite{Schlichting1968,Lagree2005}) and $f_2H := \alpha_1\varphi'(0)$. 

Using definitions \eqref{delta2fnctdelta1} and \eqref{parietalwithprofile}, we can rewrite the Von K\'arm\'an equation (\ref{Equedelta}) in following form
\begin{equation}\label{VKdelta1}
  \partial_t(u_e\delta_1)+u_e\delta_1\partial_xu_e+\partial_x\left(\frac{u_e^2\delta_1}{H}\right)=\dfrac{f_2 H}{\delta_1}u_e.
\end{equation}
At this stage, choosing a velocity profile in the viscous layer amounts to impose a closure formula on the shape factor $H$ and the friction factor $f_2$. Once this is done, the von K\'arm\'an will be given 
a closed form, in terms of $u_e$ and the displacement thickness $\delta_1$.

Several shapes can be used for the profile, including turbulent ones. As far as laminar profiles are concerned, we refer to \cite[ch.X]{Schlichting1968} for elements of comparisons between different profiles. We shall assume that $\varphi$ depends solely on the variable $\xi$ according to the {\em similarity principle} \cite{Schlichting1968} on velocity profile over a flat plane at zero-incidence. For the sake of clarity, let us briefly present in what follows several classical profile functions in viscous layer from which we establish some instructive laws on $H$ and $f_2$ for our study. 

\medskip
\noindent{\bf Pohlhausen polynomial profile.} This kind of approach is known as the Pohlhausen solution which consists in considering a polynomial approximation of velocity profile. The polynomial coefficients are chosen such that $\varphi(\xi)$ verifies boundary condition \eqref{profileProp}. The first and also the simplest case consists in a linear profile
\begin{equation*}
  \varphi(\xi) = \xi, \quad H = 3, \quad f_2 = 0.167.
\end{equation*}
Higher order polynomial can be derived by imposing additional conditions $0 = \varphi'(1) = \varphi''(1) = \cdots$ which mean that the transition between the viscous layer and the inviscid layer must be smooth. Thus, the second and third-order profiles write
\begin{align*}
  \varphi(\xi) & = 2\xi - \xi^2, \quad H = 2.5, \quad f_2 = 0.267, \\
  \varphi(\xi) & = \frac{3}{2}\xi - \frac{1}{2}\xi^3, \quad H = 2.7, \quad f_2 = 0.208.
\end{align*}
We notice that although these profiles are quite different, both lead to very close values of $H,f_2$. Nevertheless, they do not allow the observation of separation for decelerated flows. This difficulty can be overcome only by using a fourth-order polynomial with a free parameter $\Lambda$. The resulting profile function, named Pohlhausen4 in the following, takes the form
\begin{equation*}
  \varphi(\xi) = (2\xi - 2\xi^3 + \xi^4) + \frac{\Lambda}{6}\xi(1 - \xi)^3.
\end{equation*}
The parameter $\Lambda$ is related to the pressure gradient, and therefore to the variation of velocity of the inviscid flow. Indeed, taking the second-order derivative $\varphi''(0)$ and by evaluating Prandtl 
momentum equation \eqref{Eccl2} at $\bar y = 0$ we deduce 
\begin{equation*}
  \Lambda = - \varphi''(0) = -\frac{\Delta^2}{u_e}\partial_{\bar y}^2\bar u|_{\bar y = 0} = -\frac{\Delta^2}{u_e}\partial_x p = \Delta^2\partial_x u_e,
\end{equation*}
where the last equality is the stationary version of the momentum equation \eqref{eq:Emomentum} in the inviscid layer. Since by definition $\varphi(\xi) \leq 1$, we have $\Lambda \leq 12$, and we emphasize that for $\Lambda \leq -12$, the velocity profiles exhibit negative regions that correspond to reverse flow (see Figure \ref{fig:Pohlhausen-profile}, left side). 
\begin{figure}[!ht]
  \centering
    \includegraphics[width=0.49\linewidth]{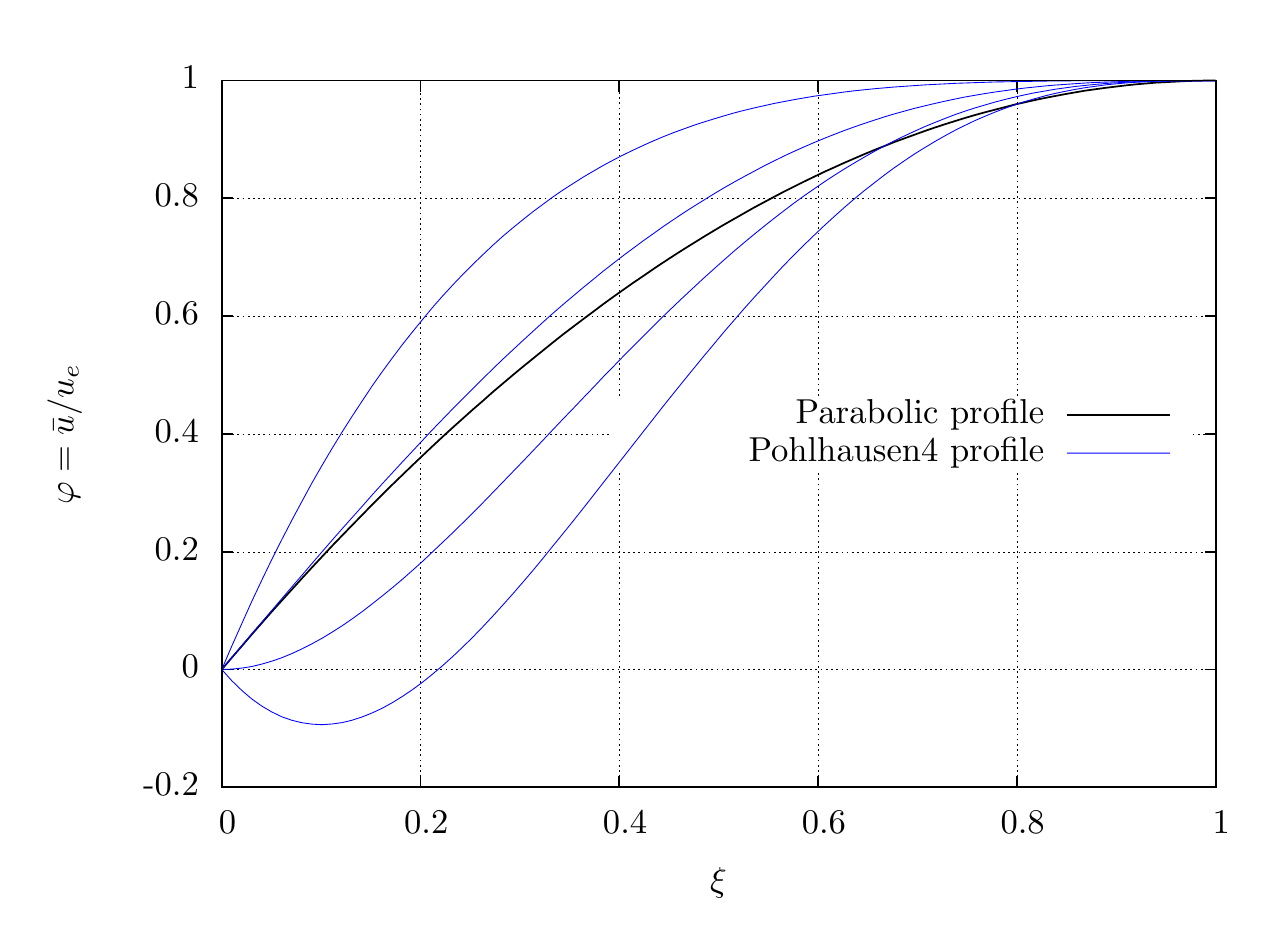}
    \includegraphics[width=0.49\linewidth]{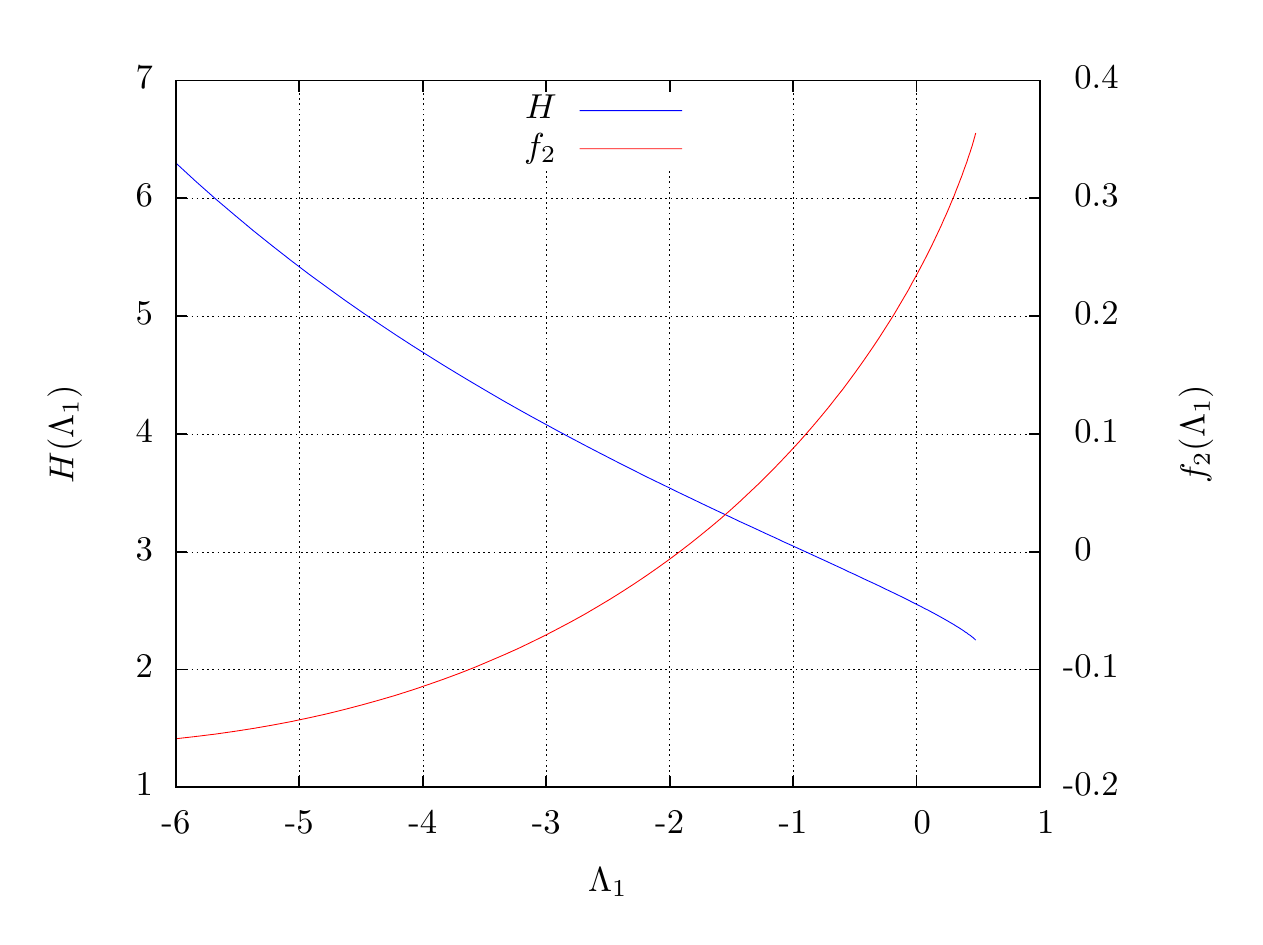}
    \caption{{\small Polynomial approximation of the velocity profile: (left) parabolic (black) vs Pohlhausen of order 4 profiles with $\Lambda = 12,0,-12,-24$ (blue); (right) closure on the shape factor $H$ and and the friction $f_2$ factors based on Pohlhausen of order 4. Note that reverse flows $f_2<0$ are possible (color online).}}
    \label{fig:Pohlhausen-profile}
 \end{figure}

We turn now to study the shape and friction factors based on the profile under consideration. First, substituting the fourth-order polynomial into definitions \eqref{delta2fnctdelta1} and \eqref{parietalwithprofile} allows to express $H,f_2$ as explicit functions of $\Lambda$, but they are omitted here for the sake of compactness. Nevertheless we notice that these relations are just formal since the "physical thickness" $\Delta(\bar t,\bar x)$ remains unknown once the velocity $\bar u$ in viscous layer has not yet been solved. In practice, it is more convenient to replace $\Delta$ by the displacement thickness $\delta_1$. A possible way, according to   \cite{Lagree2005,Lagree2005a}, is to introduce a new parameter $\Lambda_1$, inspired from the definition of $\Lambda$, given by
\begin{equation}\label{eq:Lambda1}
  \Lambda_1 := \delta_1^2 \partial_x u_e = \left(\frac{36-\Lambda}{120}\right)^2\Delta^2\partial_x u_e = \left(\frac{36-\Lambda}{120}\right)^2\Lambda
\end{equation}
in which we have used the relation $\delta_1/\Delta = (36-\Lambda)/120$ obtained by substituting the fourth-order polynomial into definition of $\delta_1$. Moreover, equation \eqref{eq:Lambda1} leads to $\Lambda_1$ being monotone on the physical range $\Lambda \leq 12$. Consequently, the factors $H,f_2$ can also be expressed as functions of $\Lambda_1$. We represent on Figure \ref{fig:Pohlhausen-profile}, on the right side, the functions $H(\Lambda_1)$ and $f_2(\Lambda_1)$ with $\Lambda$ ranging from $-24$ to $12$ that corresponds to $-6 \leq \Lambda_1 \leq 0.48$. Finally, we present in Tab. \ref{tab:specific-Pohlhausen-solution} values of $H$ and $f_2$ corresponding to special cases: $\Lambda = 12$ (limit of physical range), $\Lambda = 0$ (no pressure gradient namely Blasius solution) and $\Lambda = -12$ (incipient separation).
\begin{table}[!ht]
  \centering
    \begin{tabular}{|c|c|c|c|c|}
      \hline
      Case & $\Lambda$ & $\Lambda_1$ & $H$ & $f_2$ \\
      \hline
      Limit case & 12 & 0.48 & 2.25 & 0.356 \\
      \hline 
      Blasius case & 0 & 0 & 2.554 & 0.235 \\ 
      \hline
      Incipient separation & -12 & -1.92 & 3.5 & 0 \\
      \hline
    \end{tabular}
    \caption{Specific solutions of Pohlhausen of order four profile.}
  \label{tab:specific-Pohlhausen-solution}
\end{table}

\medskip
\noindent{\bf Falkner Skan profile.} Polynomial profiles, despite their simplicity,  are rather artificial. 
Their construction is based only on some suitable boundary conditions. An alternative approach, that might be more interesting, is to use exact solutions of boundary layer equations in order to establish more physical closures. Such an approach can be done by employing the solution to Falkner-Skan equation \cite{Falkner1931}. It plays an important role to illustrate the main physical features of boundary layer phenomena. This solution describes the form of an external laminar boundary layer of a flow over a wedge. The Blasius solution for a flat plate is a particular case of this solution. Falkner-Skan equation consists of a third-order boundary value problem whose resolution is still complicate (see e.g. \cite{Cebeci1971, Zhang2009b, Keshtkar2013}). 

We do not present here any details on the resolution of Falkner-Skan equation but focus on the construction of closure formul\ae\ and compare the obtained results with those given by Pohlhausen4. First, we solve the Falkner-Skan equation on the whole physical range of pressure gradient, corresponding to the case of accelerated, decelerated and reverse flows, to obtain all the values of the triplet $(\Lambda_1, H, f_2)$. Next, we find out a numerical relation between these parameters by inspiring from the approach presented for Pohlhausen4. On Figure \ref{fig:Falkner-Skan-profile}, it is found that the Pohlhausen4 closure, although its purely algebraic derivation, presents a good agreement with Falkner-Skan when $\Lambda_1 \geq 0$ corresponding to accelerated flows. In particular, exact value of Blasius solution $(\Lambda_1 = 0, H = 2.59, f_2 = 0.22)$ is very close to that given by Pohlhausen4, see again Tab. \ref{tab:specific-Pohlhausen-solution}. However,  these closures diverge in regions of decelerated and reverse flows $(\Lambda_1 < 0)$. Incipient separation ($f_2=0$) is reached at $(\Lambda_1 = -1.09, H = 4)$ while it is $(\Lambda_1 = -1.92, H = 3.5)$ for Pohlhausen of order 4. Finally, a value $(\Lambda_1 = 0.6,H = 2.074)$ is found as limit of physical range of Falkner-Skan solution.
\begin{figure}[!ht]
  \begin{center}
    \includegraphics[width=0.49\linewidth]{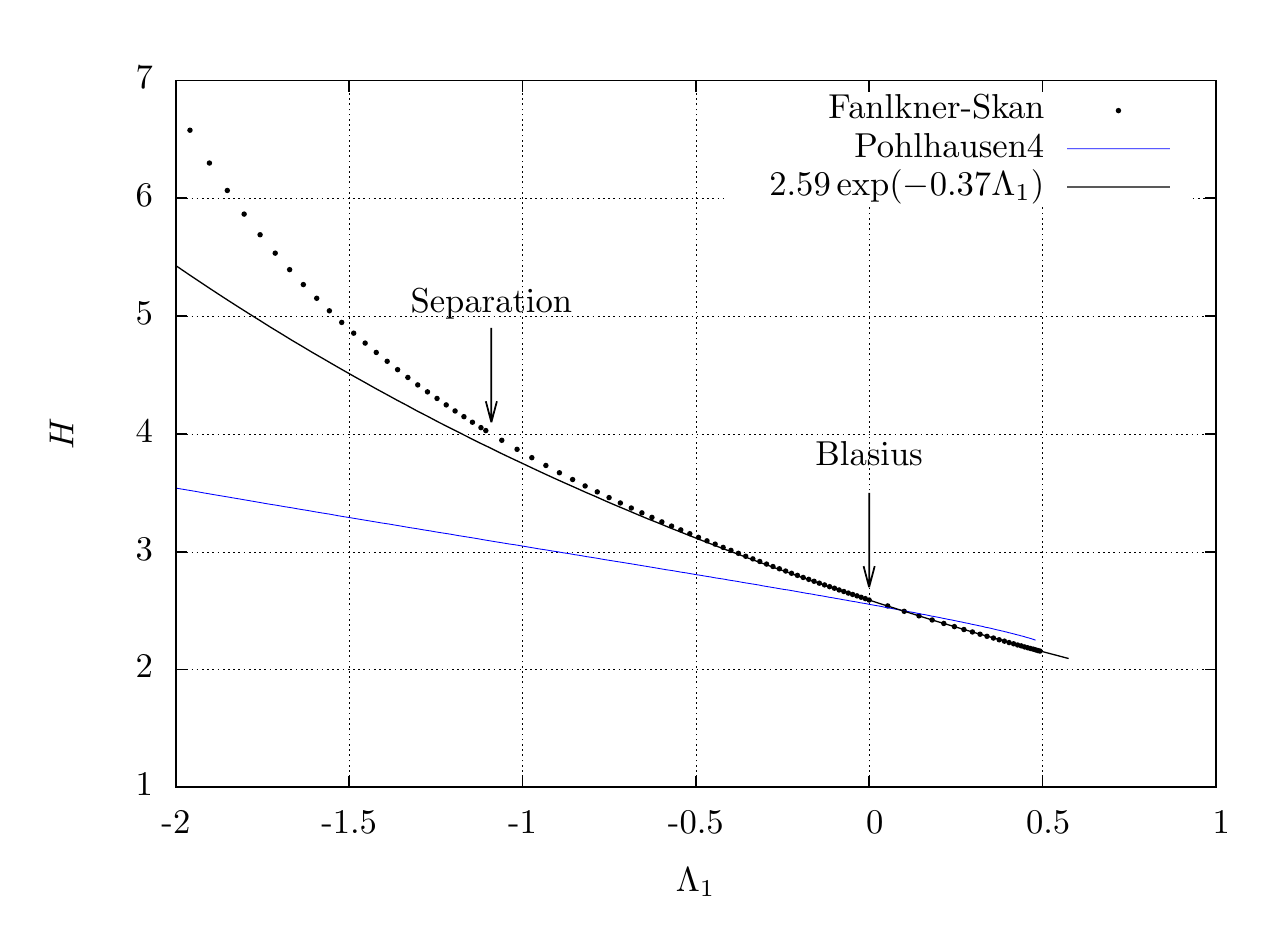}
    \includegraphics[width=0.49\linewidth]{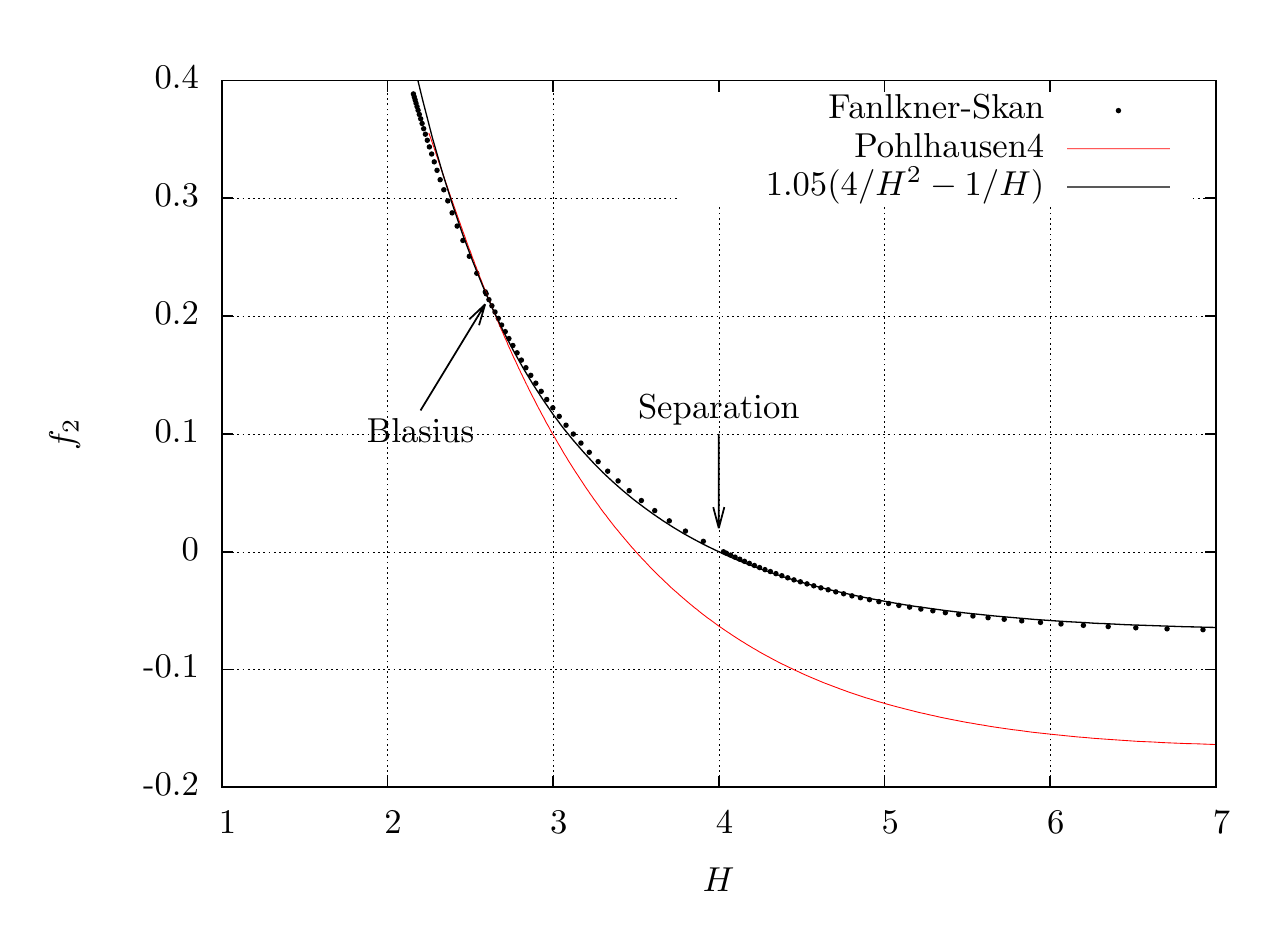}
    \caption{Falkner-Skan vs Pohlhausen order 4 closures: shape factor $H$ (left) as a function of 
     $\Lambda_1 = \delta_1^2 \partial_x u_e$ and friction factor $f_2$ (right) as a function of $H$.}
    \label{fig:Falkner-Skan-profile}
  \end{center}  \end{figure}

Lagr\'ee and Lorthois \cite{Lagree2005} proposed the following {\em ad-hoc} closure based on a fitting of Falkner-Skan solution
\begin{equation}\label{eq:Falkner-Skan-closure}
  H = \left\{
  \begin{array}{ll} 
    2.59 e^{-0.37\Lambda_1} & \text{if } \Lambda_1 < 0.6, \\
    2.074 & \text{otherwise},
  \end{array} \right.
  \text{and } f_2 = 1.05\left(\frac{4}{H^2}-\frac{1}{H}\right).
\end{equation}

As we can see on figure \ref{fig:Falkner-Skan-profile}, this numerical law presents a good agreement near Blasius solution (both accelerated and decelerated flows). These regions are also the most concerned cases in river hydraulic application (i.e. with small bed perturbation). Especially, when plotting $f_2$ as function of $H$, an excellent agreement is found. This is why we adopt \eqref{eq:Falkner-Skan-closure} for our numerical study of the present model. 

Finally, it is important to emphasize that both polynomial and numerical closure for $H$ and $f_2$ are based on steady solutions of boundary-layer equations.

\section{Extended shallow water model}\label{sec:ESW}

We are now in position to obtain the extended model we are looking for. Depth-integration the mass and momentum conservation equations of RNSP system yields shallow water equations \eqref{eq:sw-mass} and \eqref{eq:sw-momentum}, as presented in Sec. \ref{subsec:SW}. Compared to that of classical shallow water model \cite{Gerbeau01}, we have noticed, on the  one hand, the dependence in $\delta_1, \delta_2$ of the momentum equation, and on the other hand, the rise of parietal constraints in the right-hand side at order $1$ in $\bar\delta$. This motivates a new closure for the momentum flux and so the system has to be coupled with the von K\'arm\'an equation presented above.

\subsection{Towards the extended model} \label{ESW}
This section is devoted precisely to the coupling between depth-integrated shallow water equations and the von K\'arm\'an equation. It merely emphasizes that relation \eqref{usquarefctue} does not give any closure for the momentum flux.
This is  achieved by obtaining a closure on the momentum thickness $\delta_2$, through  the study of the viscous layer, as we did above in Section \ref{sec:ViscProf}. In what follows, we show that expression \eqref{usquarefctue} for the momentum flux is in fact the most convenient, since it takes into account the effect of the viscous layer, and we clarify clarify the role of the von K\'arm\'an equation. To this end, we start from the system of depth-integrated equations \eqref{eq:sw-mass}, \eqref{eq:sw-momentum} and ideal fluid equation \eqref{eq:FP}. We rewrite the momentum equation \eqref{eq:sw-momentum} with a generic form of the flux together with relation \eqref{parietal} on parietal constraints:
\begin{equation}\label{eq:sw-moment2}
\partial_t(hU)+\partial_x J = -h\partial_xp-\bar \delta \bar\tau_b,
\end{equation}
where $J$ is the momentum flux for which we seek a closure. We evidence now the fact that a convenient definition of $J$ allows to recover the von K\'arm\'an
equation from this system of integrated equations.

\begin{prop}\label{Close-flux}
  Let $(h,U,u_e,J)$ be solution to \eqref{eq:FP}, \eqref{eq:sw-mass} and \eqref{eq:sw-moment2}.
  Assume $\delta_1$ is defined by \eqref{tildeUfctue}. Then $\delta_1$ and $\delta_2$ solve the von K\'arm\'an equation \eqref{Equedelta}
  if and only if there holds
  \begin{equation}\label{eq:fluxJ}
    J = \big(h - \bar\delta(\delta_1+\delta_2)\big)u_e^2.
  \end{equation}
\end{prop}

\begin{proof} We start from the von K\'arm\'an equation and introduce \eqref{tildeUfctue} to obtain
  \begin{equation*}
  \partial_t(hu_e)-\partial_t(hU) + (hu_e-U)\partial_xu_e + \bar\delta\partial_x(u_e^2\delta_2) = \bar\delta\bar\tau_b.
  \end{equation*}
  To this equation we add \eqref{eq:sw-moment2}, the parietal term disappears, leading to
  \begin{equation*}
  \partial_t(hu_e) + \partial_xJ+hu_e\partial_xu_e-hU\partial_xu_e+\bar\delta\partial_x(u_e^2\delta_2) = -h\partial_xp.
  \end{equation*}
  Developing the time derivative and simplifying with \eqref{eq:FP} we obtain
  \begin{equation*}
  u_e\partial_th+\partial_xJ-hU\partial_xu_e+\partial_x\big(\bar\delta u_e^2\delta_2\big) = 0.
  \end{equation*}
  Finally, we use \eqref{eq:sw-mass} to eliminate the time derivative, regroup terms and get
  \begin{equation*}
  \partial_x\big(J-hu_eU+\bar\delta\delta_2u_e^2\big) = 0,
  \end{equation*}
  so that, up to a constant which can be taken equal to zero by considering that the flux is zero when the velocity is zero, we have
  \begin{equation}\label{eq:fluxJ2}
    J = hu_eU - \bar\delta\delta_2u_e^2,
  \end{equation}
  which together with \eqref{tildeUfctue} gives precisely \eqref{eq:fluxJ}.

  Conversely, we consider \eqref{tildeUfctue} and use successfully the mass and momentum balance equations
  \begin{align*}
    \partial_t(u_e\bar\delta\delta_1) &= \partial_t(hu_e)-\partial_t(hU) = u_e\partial_th+h\partial_t u_e + \partial_x\big(\big(h-\bar\delta(\delta_1+\delta_2)\big)u_e^2\big)+h\partial_xp+\bar\delta\bar\tau_b\\
    &= -u_e\partial_x(hU) - hu_e\partial_xu_e-h\partial_xp+\partial_x(hu_e^2)-\bar\delta\partial_x\big((\delta_1+\delta_2)u_e^2\big)    +h\partial_xp+\bar\delta\bar\tau_b\\
    &=-u_e\partial_x(hu_e)+u_e\partial_x(\bar\delta\delta_1u_e) - hu_e\partial_xu_e + \partial_x(hu_e^2)-\bar\delta\partial_x\big((\delta_1+\delta_2)u_e^2\big) +\bar\delta\bar\tau_b\\
    &= - \bar\delta\left(-u_e\partial_x(\delta_1u_e)+\partial_x((\delta_1+\delta_2)u_e^2)-\bar\tau_b\right).
  \end{align*}
  Noting that $\partial_x(\delta_1u_e^2)=u_e\partial_x(\delta_1u_e)+\delta_1u_e\partial_x u_e$ we recover as required the von K\'arm\'an equation.
\end{proof}

In the above proposition, we only use the three equations \eqref{eq:FP}, \eqref{eq:sw-mass}, \eqref{eq:sw-moment2} and definition (\ref{tildeUfctue}) of the displacement thickness. Replacing the physical definition (\ref{tildeUfctue}) by von K\'arm\'an equation (\ref{Equedelta}) will give back the physical definition in the sense of characteristics as explained in the next proposition.
\begin{prop}\label{Caract-delta}
  Let $J$ be defined by \eqref{eq:fluxJ}, and $(h,U,\delta_1,u_e)$ be a (smooth) solution to the system of equations \eqref{eq:FP}, \eqref{eq:sw-mass} and \eqref{eq:sw-moment2} together with von K\'arm\'an equation
  \eqref{Equedelta}. Denoting by $\delta_1^*$ the thickness obtained using \eqref{tildeUfctue}, we have 
  \begin{equation*}
  \partial_t\left(u_e(\delta_1-\delta_1^*)\right) - u_e\partial_x\left(u_e(\delta_1-\delta_1^*)\right) = 0.
  \end{equation*}
  In other words the error between these displacement thickness is constant along the characteristics of the ideal fluid. Hence, if initially $\delta_1=\delta_1^*$ then
  it is true for all times.
\end{prop}
\begin{proof}
  From \eqref{tildeUfctue}
  \begin{align*}
    \partial_t(u_e\bar\delta\delta_1^*) &= \partial_t(hu_e)-\partial_t(hU) = u_e\partial_th+h\partial_t u_e + \partial_x\big(\big(h-\bar\delta(\delta_1+\delta_2)\big)u_e^2\big)+h\partial_xp+\bar\delta\bar\tau_b\\
    &=-u_e\partial_x(hU) - hu_e\partial_xu_e-h\partial_xp+\partial_x(hu_e^2)-\bar\delta\partial_x\big((\delta_1+\delta_2)u_e^2\big)    +h\partial_xp+\bar\delta\bar\tau_b\\
    &=-u_e\partial_x(hu_e)+u_e\partial_x(\bar\delta\delta_1^*u_e) - hu_e\partial_xu_e + \partial_x(hu_e^2)-\bar\delta\partial_x\big((\delta_1+\delta_2)u_e^2\big) +\bar\delta\bar\tau_b\\
    &= - \bar\delta\left(-u_e\partial_x(\delta_1^*u_e)+\partial_x((\delta_1+\delta_2)u_e^2)-\bar\tau_b\right).
  \end{align*}
  Now using the von K\'arm\'an equation to eliminate $\delta_2$, we obtain
  \begin{align*}
    \bar\delta\partial_t(u_e\delta_1^*) &= - \bar\delta\left(-u_e\partial_x(\delta_1^*u_e)+\partial_x(\delta_1 u_e^2) -\partial_t(u_e\delta_1)-u_e\delta_1\partial_xu_e\right)  \\
    &= - \bar\delta\left(-u_e\partial_x\big(u_e(\delta_1^*-\delta_1)\big)-\partial_t(u_e\delta_1)\right),
  \end{align*}
  which is the desired result.
\end{proof}

In summary, 
propositions \ref{Close-flux} and \ref{Caract-delta} show that there are two equivalent possibilities to compute $\delta_1$ in order to close the system.
The first one consists in adding the algebraic relation \eqref{tildeUfctue}, which gives somewhat an equation of state.
The second one makes use of the von K\'arm\'an equation that we can also rewrite in the following form
\begin{equation}\label{VKprof}
  \partial_t(\delta_1u_e)+\partial_x\left((1+\frac{1}{H})\delta_1u_e^2\right)=\bar\tau_b + u_e\partial_x(\delta_1u_e),
\end{equation}
emphasizing the fact that $\delta_1u_e$ is advected with velocity $u_e$. This will be useful in particular for numerical purposes.

Putting together the results of Proposition \ref{Close-flux} and the closure formul\ae\ \eqref{delta2fnctdelta1} and \eqref{parietalwithprofile} on the velocity profile in the viscous layer, we can now rewrite momentum equation \eqref{eq:sw-momentum}---or \eqref{eq:sw-moment2}---of the viscous shallow water model in the form 
\begin{equation}\label{eq:sw-momentum-final}
\partial_t(hU)+\partial_x\left(\big(h-\bar \delta \delta_1(1+\frac{1}{H})\big)u_e^2 + \frac{h^2}{2Fr^2}\right)
  = -\frac{hf_b'}{Fr^2}-\bar\delta\bar\tau_b.
\end{equation}
Moreover, the relations between $h$, $U$, $u_e$ and $\delta_1$ through the displacement thickness \eqref{tildeUfctue} and \eqref{usquarefctue} directly imply the following expressions for the momentum flux
\begin{equation*}
  \int_{f_b}^{\eta}u^2\d{y} = (h-\bar\delta(\delta_1+\delta_2))u_e^2= hU^2 + \bar\delta(\delta_1-\delta_2-\bar\delta\delta_1^2/h)u_e^2.
\end{equation*}
The last expression clearly emphasizes that we are able to compute a non constant Boussinesq
coefficient, indeed from \eqref{eq:Boussinesq} we have
\begin{equation}\label{eq:Boussinesq2}
\beta = 1 + \bar\delta \frac{(\delta_1-\delta_2-\bar\delta\delta_1^2/h)u_e^2}{hU^2} 
 = 1 + (1 - \frac{1}{H})\frac{\bar\delta\delta_1}{h} + O(\bar\delta^2).
\end{equation}
For $\bar\delta=0$ we recover the classical shallow water system with Boussinesq coefficient equal to 1. As soon as viscosity effects arise, that is $\bar \delta>0$, we have not only the friction term on the right-hand side of \eqref{eq:sw-momentum-final} but also a correction of the same order one in $\bar\delta$ to the hydrostatic pressure. Notice that in \cite{Richard2012} a similar correction in the flux of the shallow water system is proposed to improve the study of roll-waves.
In the context of a thin viscous layer $(\bar\delta\delta_1/h \ll 1)$ we consider here, one can observe from \eqref{eq:Boussinesq2} that the Boussinesq coefficient is very close to unity, so that its impact on the the velocity correction is negligible. 
Therefore we focus on the study of the friction term.

\subsection{Final equivalent ideal fluid formulations}
The aim of this section is to propose extended shallow water (ESW) models involving an ideal fluid with velocity $u_e$ and to couple them with the von K\'arm\'an equation describing evolution of the displacement thickness $\delta_1$. The first step towards these models consists in providing two systems where we keep equation \eqref{eq:sw-mass} for mass conservation and constitutive relation \eqref{tildeUfctue} on the displacement thickness. With these relations, momentum equation \eqref{eq:sw-momentum-final} can easily be reformulated as
\begin{equation*}
h\left(\partial_tu_e+u_e\partial_xu_e+\frac{1}{Fr^2}(\partial_xh+f'_b)\right)
-\bar\delta\left(\partial_t(\delta_1u_e) + \partial_x\left((1+\frac{1}{H})\delta_1u_e^2\right)- u_e\partial_x(u_e\delta_1) -\bar\tau_b\right)=0.
\end{equation*}
It is clear on this formulation that among the three following equations: inviscid momentum \eqref{eq:FP}, viscous momentum \eqref{eq:sw-momentum-final} and von K\'arm\'an \eqref{VKprof}, once two equations are satisfied then so is the third one. 

From this consideration, the ESW can be expressed with two equivalent formulations: the first one describes an ideal fluid living on a viscous layer, which becomes some apparent topography; the second one represents an equivalent ideal fluid over the whole depth $h$. It turns out that this last formulation is the most suitable for numerical studies. Here these systems are obtained by straightforward manipulations of the equations, but it is noteworthy that they can be derived as well from the Euler system with appropriate boundary conditions, as we shall see below.

\medskip
\noindent{\bf Apparent topography formulation.} Let us derive the first system which describes an ideal fluid lying above the viscous layer of thickness $\bar\delta\delta_1$, see again figure \ref{Fig:definition-delta1}. First, we define the effective depth ${\mathcal H}:=h-\bar\delta\delta_1$ of ideal fluid. The constitutive relation \eqref{tildeUfctue} allows to write the flux of mass $hU={\mathcal H}u_e$. Next, multiplying the ideal fluid equation \eqref{eq:FP} by $\mathcal{H}$ and substituting the mass balance equation \eqref{eq:sw-mass} into it, the resulting equation together with the mass balance equations and the von K\'arm\'an equation form the following coupled system
\begin{equation}\label{eq:ESW-apparent-topo}
\left\{\begin{aligned}
 & \partial_t\mathcal{H}+\partial_x(\mathcal{H}u_e) + \partial_t(\bar\delta\delta_1) = 0, \\
  & \partial_t(\mathcal{H}u_e)+\partial_x\bigg(\mathcal{H}u_e^2+\dfrac{\mathcal{H}^2}{2Fr^2}\bigg)+u_e\partial_t(\bar\delta\delta_1) =
  -\dfrac{\mathcal{H}}{Fr^2}\big(f'_b+\partial_x(\bar\delta\delta_1)\big), \\
  & \partial_t(\delta_1u_e)+\partial_x\left((1+\frac{1}{H})\delta_1u_e^2\right) = u_e\partial_x(\delta_1u_e) + \bar\tau_b. 
\end{aligned}\right.
\end{equation}
One can see that the two first equations of the system represent a shallow flow, of thickness $\mathcal{H}$, over a modified topography, namely $f_b + \bar\delta\delta_1$. This has to be related to the so-called ``apparent topography'' formulation, where for numerical purposes the friction term is rewritten as the derivative of some function, see \cite{Bouchut04,Bouchut2004}. Here this derivative arises in a natural way, together with additional time derivatives in the mass and momentum equations. 

As mentioned, the apparent topography formulation can be obtained as well from integration of the Euler system over the vertical, but from the modified topography $f_b + \bar\delta\delta_1$ to the free surface $\eta$. The key point is that the boundary condition on the interface between the viscous layer and the ideal fluid is a non-penetration condition
\begin{equation}\label{eq:TopAppGliss}
  v_e|_{y = f_b + \bar\delta\delta_1} = u_e(f'_b + \partial_x(\bar\delta \delta_1)).
\end{equation}

As for the classical shallow water model, it is straightforward to obtain from the first two equations of system \eqref{eq:ESW-apparent-topo} the following energy balance equation:
\begin{equation}\label{eq:ESW-energy}
    \partial_t\left(\frac{\mathcal{H}u_e^2}{2} + \frac{(\mathcal{H} + f_b + \bar\delta\delta_1)^2}{2Fr^2} \right) + \partial_x\left(u_e\left(\frac{\mathcal{H}u_e^2}{2} + \frac{\mathcal{H}(\mathcal{H} + f_b + \bar\delta\delta_1)^2}{2Fr^2} \right) \right) = -\frac{u_e^2}{2}\partial_t(\bar\delta\delta_1).
\end{equation}
In contrast with the classical shallow water system, the ``dissipation'' of energy here is driven by the dynamical behaviour in time of the displacement thickness $\delta_1$. 
If $\partial_t(\bar\delta\delta_1) \geq 0$, as for the Blasius-Stokes solution presented below in Section \ref{sec:blasius-stokes},
we have actually dissipation of energy by the bottom friction, thus some stability of the solutions. If this not the case, the problem of energy dissipation is open.

\medskip
\noindent{\bf Interactive Boundary Layer formulation.} This model describes an ideal fluid on the whole water depth $h$. Multiplying \eqref{VKprof} by $\bar\delta$ and substituting it into \eqref{eq:sw-momentum-final} leads to a model being very similar to the usual shallow water one. The resulting system reads
\begin{equation}\label{eq:ESW-IBL}
\left\{\begin{aligned}
 & \partial_th+\partial_x(hu_e -\bar \delta \delta_1u_e) = 0, \\
 & \partial_t(hu_e)+\partial_x\left(hu_e^2 + \frac{h^2}{2Fr^2}\right) = -\frac{hf_b'}{Fr^2} + u_e\partial_x(\bar\delta\delta_1u_e), \\
 & \partial_t(\delta_1u_e)+\partial_x\left((1+\frac{1}{H})\delta_1u_e^2\right) = u_e\partial_x(\delta_1u_e) + \bar\tau_b.  
\end{aligned}\right.
\end{equation}
Once the conservative variables $h, hu_e$ and $\delta_1u_e$ are solved, the averaged velocity $U$, if needed, can be recovered using relation \eqref{tildeUfctue}. An noticeable feature of this formulation is that the friction 
is no longer explicitly present in the momentum equation, it is replaced by an advection term on the ``momentum'' $\delta_1u_e$. This term actually represents the momentum exchange between these two layers. 
System \eqref{eq:ESW-IBL} enjoys a more conservative structure than the previous formulation, which makes it more suitable for numerical discretization.

This formulation of the model can also be obtained by integrating the Euler system over the whole water depth, but the non-penetration condition \eqref{eq:non-penetration} has to be replaced by a slightly modified one, 
called  {\em transpiration condition}, that writes 
\begin{equation}\label{eq:transpi}
  v_e|_{y = f_b} = u_ef'_b + \bar\delta\partial_x(\delta_1u_e).
\end{equation}
which is in fact a formulation of the ``Interactive Boundary Layer'' (IBL) or ``Viscous Inviscid Interaction'' in aerodynamics, see \cite{Lagree2010}. Comparing with \eqref{eq:non-penetration}, one can see that the transpiration condition is a correction of order one in $\bar\delta$ of the non-penetration condition due to the development of viscous layer. We shall refer to \eqref{eq:ESW-IBL} as the IBL formulation in the following.
For the sake of completeness, we briefly recall how condition \eqref{eq:transpi} is derived. 

We wish to estimate the vertical velocity $v_e$ of the ideal fluid at the bottom $y = f_b$. We first notice that, since $\partial_yu_e=0$, integrating incompressibility equation \eqref{LWmass} of ideal fluid over the whole water depth yields
\begin{equation*}
    v_e|_{y = f_b} = v_e|_{y = \eta} + (\eta-f_b)\partial_x u_e.
\end{equation*}
Next, we have on the one hand 
\begin{equation*}
    \bar\delta\bar v|_{\bar y = \bar\eta} = v_e|_{y = \eta} -f'_bu_e
\end{equation*}
from the Prandtl shift \eqref{PrandtlShift}. On the other hand the vertical velocity $\bar v$ is recovered through the incompressibility relation \eqref{Eccl1} in the viscous layer, since we can write $\partial_{\bar y}\bar v = \partial_{\bar x}(u_e-\bar u) - \partial_{\bar x}u_e$. Integrating this later equation over the depth to obtain 
\begin{equation*}
\bar\delta\bar v|_{\bar y = \bar\eta} = \bar\delta\partial_x(u_e\delta_1)-(\eta-f_b)\partial_xu_e.
\end{equation*}
Putting things together leads to the required transpiration boundary condition \eqref{eq:transpi}.

\section{Numerical analysis of the IBL formulation}\label{sec:Num}
We propose here a numerical implementation of the IBL formulation  \eqref{eq:ESW-IBL}. This model was chosen because of its relative simplicity compared to the apparent topography and to the four equations models. Also, it can be
related to the previous implementation in the context of rigid pipes, see \cite{Lagree2005}. In order to  design a finite volume solver, we first rewrite this system in vector form as
\begin{equation*}
  \partial_t W + \partial_x F(W) + B(W) = \tau(W),
\end{equation*}
where the conservative variable $W$, the flux $F(W)$, the convective term $B(W)$ and the source term $\tau(W)$ are defined by
\begin{equation*}
  W = \begin{pmatrix} 
    h \\ hu_e \\ \delta_1 u_e 
  \end{pmatrix}, \quad
  F(W) = \begin{pmatrix} 
    hu_e - \bar\delta\delta_1u_e \\
    hu_e^2 + \frac{h^2}{2Fr^2} \\
    (1+\frac{1}{H})\delta_1u_e^2
  \end{pmatrix}, \quad 
  B(W) = {}-\begin{pmatrix} 
    0 \\
    u_e\partial_x(\bar\delta\delta_1 u_e) \\
    u_e\partial_x(\delta_1 u_e)
  \end{pmatrix}, \quad
  \tau(W) = \begin{pmatrix} 
    0 \\
    -\frac{hf_b'}{Fr^2} \\
    \frac{f_2H}{\delta_1}u_e
    \end{pmatrix}.
\end{equation*}
The system is numerically solved by a splitting method. First we solve the so-called convective part
    \begin{equation}\label{eq:ESW-convection}
        \partial_t W + \partial_x F(W) + B(W) = 0,
    \end{equation}
together with the topography source term $-\frac{hf_b'}{Fr^2}$, in a well-balanced way. Next, we solve the friction part
    \begin{equation}\label{eq:ESW-friction}
        \partial_t(\delta_1 u_e) = \frac{f_2H}{\delta_1}u_e
    \end{equation}
by a semi-implicit method.

\subsection{Eigenvalue analysis}
For numerical purpose and stability analysis, we first study the hyperbolicity of the convective part \eqref{eq:ESW-convection} of the model. To this end, we rewrite \eqref{eq:ESW-convection} in quasi-linear form $\partial_t W + A(W)\partial_x W = 0$. The system is said to be hyperbolic if the convective matrix $A(W)$ is $\mathbb{R}$-diagonalizable, and strictly hyperbolic if the eigenvalues are distinct. Estimating these eigenvalues is also important to design an explicit finite volume scheme for the system. As eigenvalues are invariant by changing variables \cite{Godlewski96}, it is therefore more convenient to make the variable change $W \mapsto Y(W) := (h,u_e,\delta_1u_e)^t$ and study eigenvalues of the corresponding convective matrix
\begin{equation*}
  \tilde{A}(Y) = \begin{pmatrix} 
    u_e & h & - \bar\delta  \\
    Fr^{-2} & u_e & 0   \\
    0 & a & b  - u_e   
  \end{pmatrix}, 
\end{equation*}
where we have denoted the partial derivatives
\begin{equation*}
    a :=  \frac{\partial}{\partial u_e}\left((1+\frac{1}{H})\delta_1u_e^2\right), \qquad b := \frac{\partial}{\partial(\delta_1u_e)}\left((1+\frac{1}{H})\delta_1u_e^2\right).
\end{equation*}
Their explicit expressions can be computed as well once a closure formula for the shape factor is provided, since the shape factor $H$ is a function of $\Lambda_1$, so $H$ depends only on $u_e$ and $\delta_1u_e$. 

The characteristic polynomial of $\tilde{A}(Y)$ reads
\begin{align*}
  P(\lambda) = {\rm det}(\tilde{A} - \lambda~{\rm Id}) &= - \left[(b - u_e - \lambda)\left((u_e-\lambda)^2 - Fr^{-2}h\right) -\bar\delta Fr^{-2} a\right] \\
  &= {}-P_{SW}(\lambda) + d_{\bar\delta},
\end{align*}
where
\begin{equation*}
P_{SW}(\lambda) = (b - u_e - \lambda)\left((u_e-\lambda)^2 - Fr^{-2}h\right), \quad d_{\bar\delta} := \bar\delta Fr^{-2} a.
\end{equation*}
The polynomial $P_{SW}$ has always 3 roots, denoted $\lambda_{1,2,3}^0$, which represent the wave speeds of the system with no coupling between shallow water equations and the Von K\'arm\'an one. Indeed, the first two roots $\lambda_{1,2}^0$ express the propagation velocities of ideal fluid while the last one $\lambda_3^0$ approximate that of the viscous layer:
\begin{equation}\label{eq:SW-wave-speeds}
  \lambda_{1,2}^0 = u_e\pm \frac{\sqrt{h}}{Fr}, \quad \lambda_3^0 = b - u_e.
\end{equation}

From this it follows that, for $\bar\delta$ small enough, the characteristic polynomial admits 3 real eigenvalues and the system turns out to be hyperbolic. More precisely, denoting $\lambda_{-} < \lambda_{+}$ the roots of $P'_{SW}(\lambda)$, the system is hyperbolic when $d_{\bar\delta}$ lies between $P_{SW}(\lambda_{\pm})$, see figure \ref{Fig-hyperbolicity}, that is
\begin{equation*}
    P_{SW}(\lambda_-) < \frac{\bar\delta a}{Fr^2} < P_{SW}(\lambda_+).
\end{equation*}
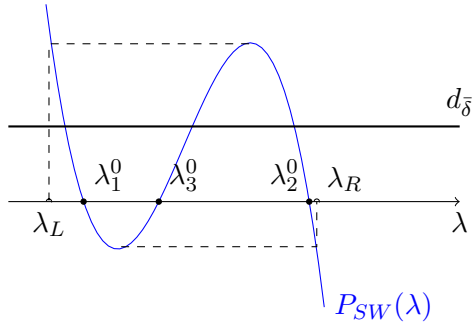
\begin{figure}[ht!]\centering
  \begin{tikzpicture}[domain=0:6, scale=1]
    \draw[->] (0,0) -- (6,0) node[below]{$\lambda$};
    \draw[blue] plot[domain=0.5:4.2,samples=50] (\x,{-(\x - 1.)*(\x - 2.)*(\x - 4.)}) node[right]{$P_{SW}(\lambda)$};
    \draw[fill=black] (1,0) circle (1pt) node[above right]{$\lambda_1^0$};
    \draw[fill=black] (2,0) circle (1pt) node[above right]{$\lambda_3^0$};
    \draw[fill=black] (4,0) circle (1pt) node[above left]{$\lambda_2^0$};
    \draw[color=black, thick] (0,1) -- (6,1) node[above]{$d_{\bar\delta}$};
    \draw[color=black, dashed] (3.2,2.1) -- (0.54,2.1) -- (0.54,0) circle (1pt) node[below]{$\lambda_L$};
    \draw[color=black, dashed] (1.5,-0.6) -- (4.1,-0.6) -- (4.1,0) circle (1pt) node[above right]{$\lambda_R$};
  \end{tikzpicture}
  \caption{Roots of characteristic polynomial of the coupled system.}
  \label{Fig-hyperbolicity} 
\end{figure}

These solutions denoted $\lambda_{1,2,3}$ are order one perturbations in $\bar\delta$ of the roots of $P_{SW}$, namely
\begin{equation*}
    \lambda_{1,2,3} = \lambda_{1,2,3}^0 + O(\bar\delta).
\end{equation*}
It is straightforward to verify that for the case of closure \eqref{eq:Falkner-Skan-closure} based on Falkner-Skan solutions, the third wave speed reads
\begin{equation}\label{eq:FS-lambda3}
\lambda_3^0 = \frac{u_e}{H}(1 + 0.74\Lambda_1), \quad \text{for } \Lambda_1 < 0.6,
\end{equation}
and in particular when $\Lambda_1 = 0$, i.e. the Blasius solution, the viscous layer propagates downstream at velocity $\lambda_3^0 = u_e/H \simeq 0.39 u_e$ (see figure \ref{fig:wave-speeds}-left).

Even within the hyperbolic regime, the eigenstructure is given only implicitly, due to the form of nonlinear coupling between the ideal fluid and viscous layer. As a result, when dealing with numerical methods requiring characteristic field decomposition, e.g. Roe type method, the eigenvalues need to be computed by numerical root finding. In the absence of analytic expressions for the eigenvectors, building desirable properties for such a scheme may be more difficult. We propose hereafter a HLL type scheme \cite{Harten83} taking advantage that the numerical flux arises directly from the governing equations \eqref{eq:ESW-convection} and only an estimation of lowest and the largest wave speeds $\lambda_{L,R}$ is required. Such a wave speeds estimation can be done by using accurate Nickalls's bounds \cite{Nickalls2011}, which writes
\begin{equation}\label{eq:estimate-wave-speeds}
 \lambda_{L,R} := \frac{1}{3}\left(u_e + b \mp 2\sqrt{(2u_e-b)^2 + \frac{3h}{Fr^2}} \right), \quad b = u_e\left(1 + \frac{1+0.74\Lambda_1}{H}\right),
\end{equation}
in which we have used \eqref{eq:FS-lambda3} provided from the Falkner-Skan closure. 

On figure \ref{fig:wave-speeds} we compare the shallow water  wave speeds $\lambda_{1,2}^0$, the viscous layer one $\lambda_3^0$ and the bounds $\lambda_{L,R}$ given by \eqref{eq:estimate-wave-speeds}. Rescaling these velocities by $u_e$, the result can be considered as functions of the {\em local} Froude number $Fr_0:=Fr\sqrt{u_e^2/h}$ and the pressure parameter $\Lambda_1$. The left figure shows the dependence on $Fr_0$ in the case of Blasius solution, i.e. $\Lambda_1 = 0$, both for subcritical and supercritial regimes. We find that the estimation $\lambda_R$ for the largest wave speed is very accurate. On the right figure, we display the dependence on $\Lambda_1$ for a fixed value of $Fr_0$, e.g. by considering the case of critical flow so $\lambda_1^0 = 0$. As we can see, the velocity $\lambda_3^0$ can be negative---the wave associated to viscous layer propagate upstream---for large reverse flow.  

\begin{figure}[!ht]
  \centering
    \includegraphics[width=0.49\linewidth]{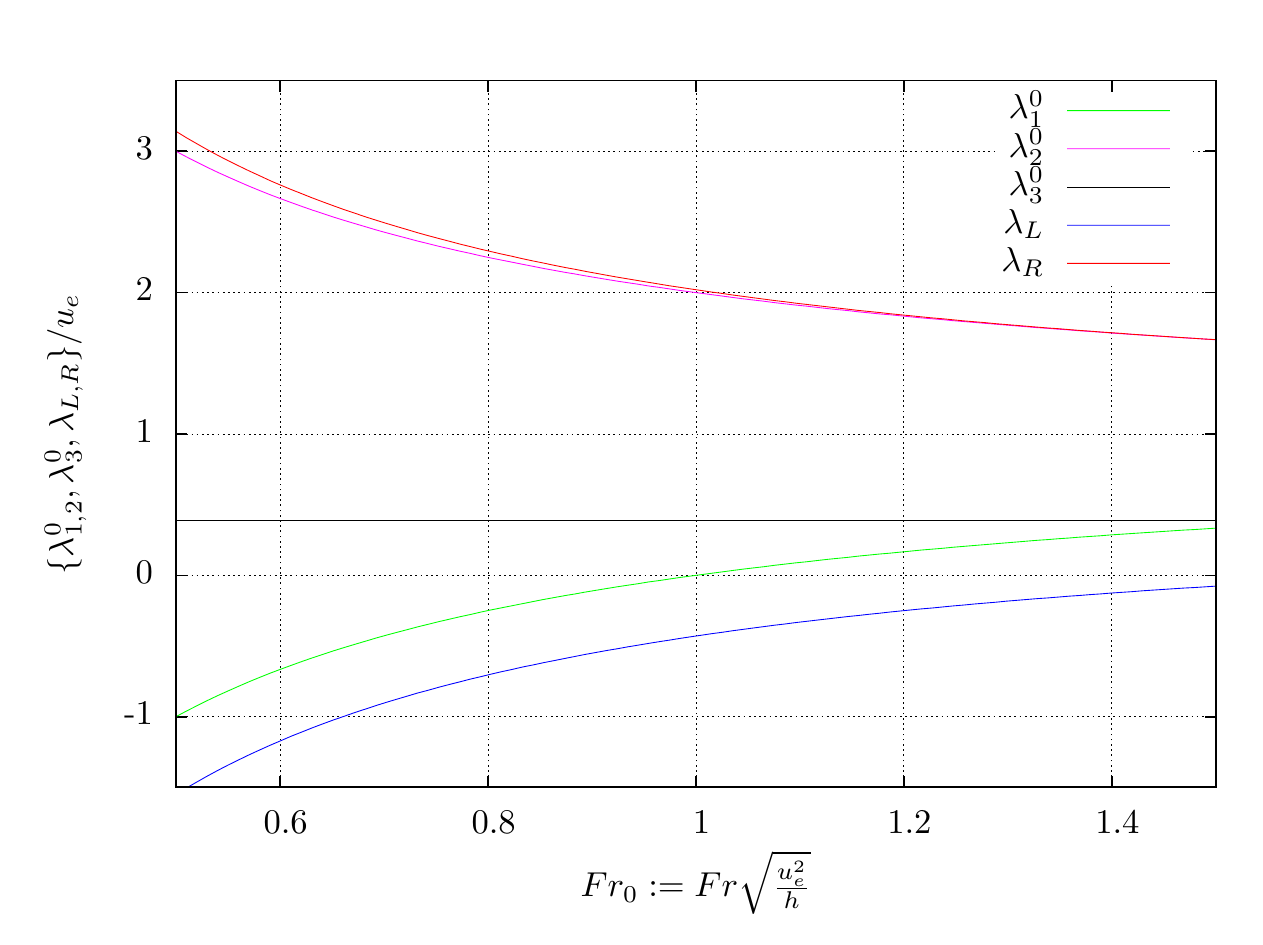}
    \includegraphics[width=0.49\linewidth]{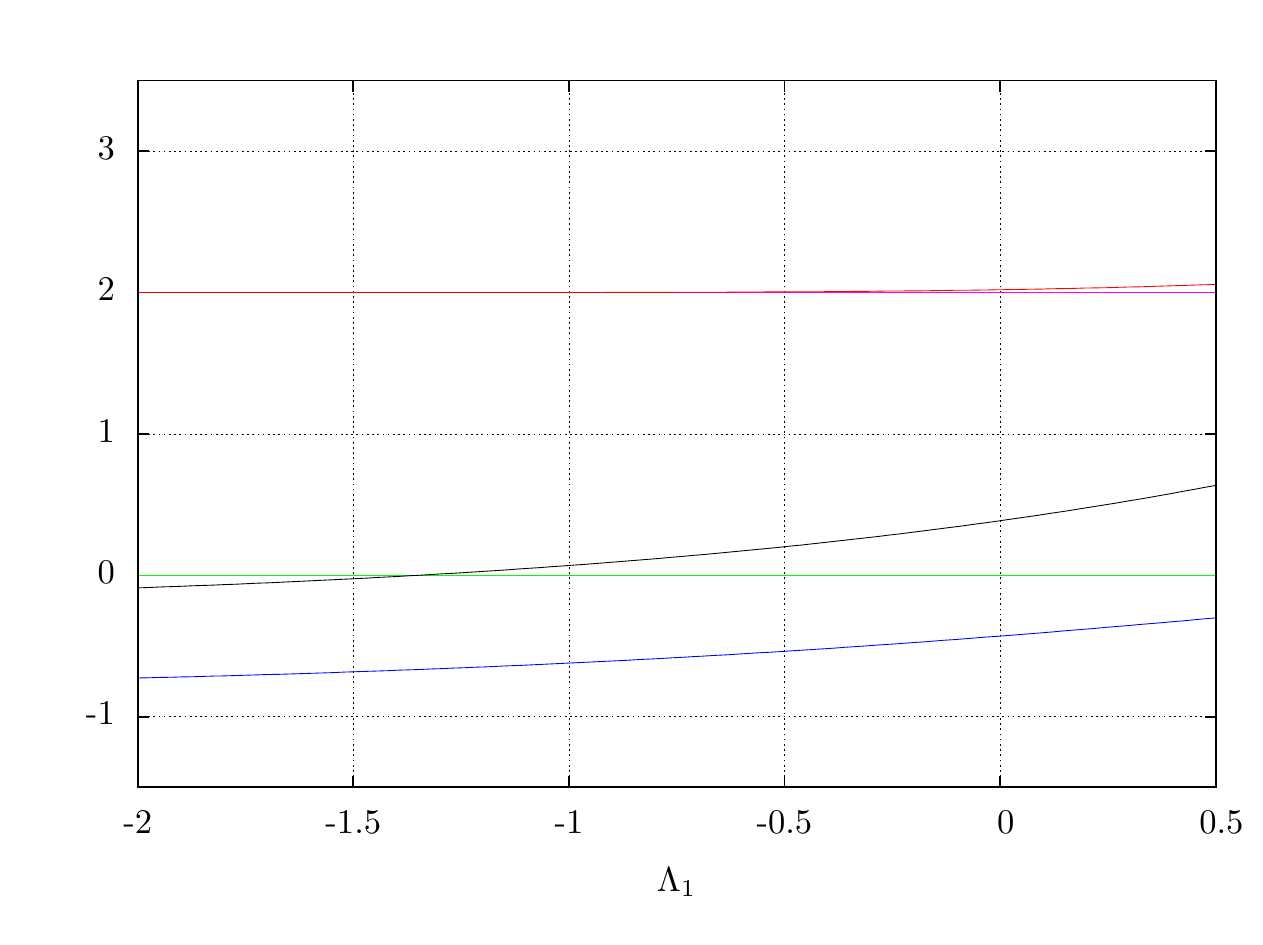}
    \caption{Comparison of wave speeds $\lambda^0_{1,2,3}$ and the estimations $\lambda_{L,R}$. Left: Blasius case ($\Lambda_1 = 0$), wave speeds as a function of the Froude nimber; Right: critical flow ($Fr_0=1$), wave speeds as a function of $\Lambda_1$ (same color code on both graphs, color online).}
    \label{fig:wave-speeds}
  \end{figure}

\subsection{A Godunov-type finite volume scheme}
Let us recall some basic notations of finite volume discretization. We introduce a space step $\Delta x$ and a time step $\Delta t$, both assumed to be constant for simplicity. The computational domain is discretized by a sequence of points $x_{j+1/2} := j\Delta x$ for $j\in\mathbb{Z}$. We define $W_j^0$ a piecewise constant approximation of initial condition on each control volume $C_j:= ]x_{j-1/2}, x_{j+1/2}[$. The time step $\Delta t$ for a mesh size $\Delta x$ has to atisfiy the well-known CFL condition
\begin{equation}\label{eq:CFL}
  \Delta t \leq \frac{\Delta x}{2|\lambda_{\rm max}|},
\end{equation}
where $\lambda_{\rm max}$ is the largest eigenvalue expressing the fastest wave speed of the system. This condition ensures that information of each Riemann problem at a cell's interface does not cross more than one cell.

\medskip
\noindent{\bf Convection step.} Assume that the solution $W_j^n$ at time $t^n$ is known. Godunov-type schemes compute the solution to \eqref{eq:ESW-convection} at the next time level $t^{n+1}:= t^n + \Delta t$ by building first an approximate solution $W_{\Delta}(x,t)$ of the Riemann problem at each interface $x_{j+1/2}$ with initial data $\{W_j^n\}_{j\in \mathbb{Z}}$, and next averaging $W_{\Delta}(x,\Delta t)$ on each control volume to obtain a piecewise constant solution $W_j^{n+1/2}$.

  Given initial data $(W_L,W_R)$ of local Riemann problem, we adopt a simple Riemann solver $W_{\Delta}(x,t)$ composed by three discontinuity waves propagating with velocities $\lambda_L\leq 0$, $\lambda_0 = 0$, $\lambda_R \geq 0$ and two intermediate states $W_L^*$ and $W_R^*$ in the star region, see figure \ref{fig:Riemann-problem}. 
  In the following $\lambda_{L,R}$ are given by formula \eqref{eq:estimate-wave-speeds}, and the zero velocity $\lambda_0$ corresponds to the stationary contact discontinuity associated with the topography. The third eigenvalue $\lambda_3$ 
  was not used here because the analytical expression of the Riemann invariant is unknown.
  
  Under CFL condition \eqref{eq:CFL}, the first-order three-points finite volume scheme writes
  \begin{equation}\label{eq:Godunov-scheme}
    W_j^{n+1/2} = W_j^n - \frac{\Delta t}{\Delta x}\left(F_{j+1/2}^L - F_{j-1/2}^R\right),
  \end{equation}
  where the left- and right- numerical fluxes $F_{j+1/2}^{L,R}  := F^{L,R}(W_j^n,W_{j+1}^n)$ are given by
  \begin{equation}\label{eq:numerical-fluxes}
    \begin{aligned}
      & F^L(W_L,W_R)  := F(W_L) + \lambda_L(W_L^*-W_L), \\
      & F^R(W_L,W_R)  := F(W_R) - \lambda_R(W_R-W_R^*).
    \end{aligned}
  \end{equation}
  Therefore, designing  such a scheme consists in determinating the intermediate states $W_{L,R}^*$ in the star region.

  \begin{figure}[ht!]
    \centering
    \begin{tikzpicture}[scale=0.7]
      \draw[->] (-5,0) -- (5,0)   node[above] {$x$};
      \draw[->] (0,0) -- (0,3.5)  node[left] {$t$};
      \draw[fill=black] (0,0) circle (2pt) node[below]{0};
      \draw[dashed] (-4.5,2.5) -- (4.5,2.5) node[right] {$\Delta t$};
      \draw[dashed] (-4,2.5) -- (-4,0) node[below]{$-\Delta x/2$};
      \draw[dashed] (4,2.5) -- (4,0) node[below]{$\Delta x/2$};
      \draw[color=blue] (0,0) -- (0,2.5) node[above right]{$\lambda_0$};
      \draw[color=blue] (0,0) -- (3.5,2.5) node[above] {$\lambda_R$};
      \draw[color=blue] (0,0) -- (-3,2.5) node[above] {$\lambda_L$};
      \draw (-2.5,0.7) node{$W_L$};
      \draw (2.5,0.7) node{$W_R$}; 
      \draw (-0.3,1.7) node[left] {$W_L^*$}; 
      \draw (0.3,1.7) node[right] {$W_R^*$};  
      \end{tikzpicture}
    \caption{A three-waves approximate Riemann problem.}
    \label{fig:Riemann-problem} 
  \end{figure}
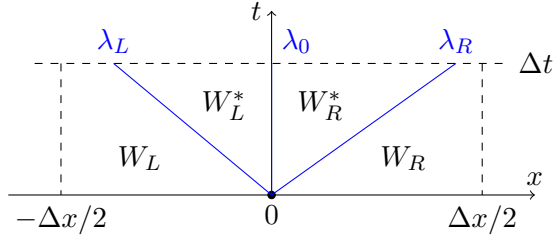

  According to \cite{Harten83} the approximate solver $W_{\Delta}(x,t)$ must be consistent with the exact solution $W_{\mathcal{R}}(x,t)$ in the sense that
  \begin{equation*}
    \int_{-\frac{\Delta x}{2}}^{\frac{\Delta x}{2}}W_{\mathcal{R}}(x,\Delta t)\d{x} = \int_{-\frac{\Delta x}{2}}^{\frac{\Delta x}{2}}W_{\Delta}(x,\Delta t)\d{x}.
  \end{equation*}
  Applying to the conservation law \eqref{eq:ESW-convection}, this integral consistency condition provides that the intermediate states satisfy the following relations
  \begin{align}
    \lambda_Rh_R^* - \lambda_Lh_L^* & = \lambda_Rh_R - \lambda_Lh_L - \jump{hu_e -  \bar\delta\delta_1 u_e}, \label{eq:consistency-h}\\
    \lambda_R(hu_e)_R^* - \lambda_L(hu_e)_L^* & = \lambda_R(hu_e)_R - \lambda_L(hu_e)_L - \jump{hu_e^2+\frac{h^2}{2Fr^2}} - \Delta x\average{\frac{hf_b'}{Fr^2}} + \bar\delta\Delta x\average{u_e\pd{}{x}(\delta_1 u_e)}, \label{eq:consistency-hu}\\
    \lambda_R(\delta_1u_e)_R^* - \lambda_L(\delta_1u_e)_L^* & = \lambda_R(\delta_1u_e)_R - \lambda_L(\delta_1u_e)_L - \jump{(1+\frac{1}{H})\delta_1 u_e^2} + \Delta x\average{u_e\pd{}{x}(\delta_1 u_e)}, \label{eq:consistency-du}
  \end{align}
  where $\jump{\bullet} := (\bullet)_R - (\bullet)_L$ standings for the usual jump operator and we have introduced the space-time averaging of the non-conservative source terms
  \begin{align*}
    \average{\frac{hf_b'}{Fr^2}} & = \frac{1}{\Delta x\Delta t}\int_{-\frac{\Delta x}{2}}^{\frac{\Delta x}{2}}\int_0^{\Delta t} \frac{hf_b'}{Fr^2}\d{t}\d{x}, \\
    \average{u_e\pd{}{x}(\delta_1 u_e)} & = \frac{1}{\Delta x\Delta t}\int_{-\frac{\Delta x}{2}}^{\frac{\Delta x}{2}}\int_0^{\Delta t} u_e\pd{}{x}(\delta_1 u_e)\d{t}\d{x}.  
  \end{align*}

  It is well-known that the non-conservative products arising in the source terms may not make sense as distributions. It is possible to give a rigorous definition using Vol'pert's calculus on BV functions \cite{Volpert1967}. This choice leads to following approximations
  \begin{equation}\label{eq:source-term-approximation}
    \Delta x\average{\frac{hf_b'}{Fr^2}} = \frac{h_L+h_R}{2Fr^2}\jump{f_b}, \quad \Delta x\average{u_e\pd{}{x}(\delta_1 u_e)} = \frac{(hu_e)_L + (hu_e)_R}{h_L+h_R}\jump{\delta_1 u_e},
  \end{equation}
  which preserve at least the \emph{lake-at-rest} equilibrium state, that is $u=0, \delta_1 = 0, \jump{h+f_b} = 0$. 
  
  Consistency conditions \eqref{eq:consistency-h}--\eqref{eq:consistency-du} have to be complemented by three additional relations in order to solve completly the intermediate states. The Riemann invariants associated to the stationary contact wave can be used to provide these missing relations
  \begin{gather}
    (hu_e)_L^* - \bar\delta(\delta_1u_e)_L^* = (hu_e)_R^* - \bar\delta (\delta_1u_e)_R^* , \label{eq:bernoulli-hu}\\
    \left(\frac{(u_e)_R^{*2}}{2} + \frac{h_R^*}{Fr^2}\right) - \left(\frac{(u_e)_L^{*2}}{2} + \frac{h_L^*}{Fr^2}\right) =  - \frac{\jump{f_b}}{Fr^2}, \label{eq:bernoulli-u}\\
    (\delta_1u_e)_L^* = (\delta_1u_e)_R^*.\label{eq:bernoulli-du}
  \end{gather}

  Note that the first two equations are the analytical Riemann invariants of the system applied for intermediate states while the third one related to $\delta_1u_e$ is just an approximation. In fact, the last analytical Riemann invariant is not known explicitly due to the form of nonlinear coupling between the ideal fluid and viscous layer. From the numerical point of view, this consists in approximating $(\delta_1u_e)_{L,R}^*$ in the star region by only one averaged value $(\delta_1u_e)^*$, as in the classical HLL scheme \cite{Harten83}. Moreover, this choice leads to $ (hu_e)_L^* = (hu_e)_R^*:=q^*$ and, together with \eqref{eq:bernoulli-u}, we find again the well-known Bernoulli relations of shallow water equations
  \begin{equation}\label{eq:bernoulli-SW}
    \frac{q^{*2}}{2}\left(\frac{1}{h_R^{*2}} - \frac{1}{h_L^{*2}}\right) + \frac{h_R^* - h_L^*}{Fr^2} = -\frac{\jump{f_b}}{Fr^2}.
  \end{equation}

  Plugging \eqref{eq:bernoulli-hu}, \eqref{eq:bernoulli-du} into \eqref{eq:consistency-hu} and \eqref{eq:consistency-du} allows us to solve the discharges $q^*$ and $(\delta_1u_e)^*$ in the star region. Next, intermediate water depths $h_{L,R}^*$ are obtained from \eqref{eq:consistency-h} and \eqref{eq:bernoulli-SW}. Similarly to the case of classical shallow water model, it can be shown that the present scheme is accurate and well-balanced. We refer to \cite{Goutal2017} for more technical details and discussions.  
  
  \medskip
  \noindent{\bf Friction step.} Once the convection step \eqref{eq:Godunov-scheme} is done, the next step is to take into account the friction source term of von K\'arm\'an equation to modify the displacement thickness. This consists in solving equation \eqref{eq:ESW-friction} with initial data $W^{n+1/2}$ to get the solution $W^{n+1}$ at the next time step. To this end, we use a simple semi-implicit scheme which writes 
  \begin{gather*}
      h^{n+1} = h^{n+1/2}, \quad u_e^{n+1} = u_e^{n+1/2}, \label{eq:friction-hu} \\
      \frac{(\delta_1u_e)^{n+1} - (\delta_1u_e)^{n+1/2}}{\Delta t} = \frac{(f_2H)^n}{\delta_1^{n+1}}u_e^{n+1/2}. \label{eq:friction-delta1}
  \end{gather*}
  This discretization leads to a second-order equation for $\delta_1^{n+1}$. Under condition $(\delta_1^{n+1/2})^2 + 4(f_2H)^n\Delta t \geq 0$, which implies an additional restriction on $\Delta t$ only in the case of reverse flow
  \begin{equation}\label{eq:CFL-friction}
     \Delta t \leq -\frac{(\delta_1^{n+1/2})^2}{4(f_2H)^n} \quad\text{if}\quad f_2^n < 0,
 \end{equation}
this equation has two solutions, from which  only one is physically admissible
  \begin{equation}\label{eq:ESW-friction-scheme}
  \delta_1^{n+1} = \frac{1}{2}\left(  \delta_1^{n+1/2} + \sqrt{(\delta_1^{n+1/2})^2 + 4(f_2H)^n\Delta t}\right). 
  \end{equation}
 In practice when using the Falkner-Skan closure, condition \eqref{eq:CFL-friction} is not restrictive, the time step is rather controlled by the CFL condition \eqref{eq:CFL} of the convection step.

\subsection{ Multi Layer formulation}\label{sec:MLF}
We conclude this section by a short presentation of the so-called multi-layer Saint-Venant model proposed in \cite{Audusse2011a}, which we shall use as a reference for comparison in the next section.
The authors consider a superposition of shallow water systems each one interacting with its neighbours:
\begin{equation}\label{eq:SWML}
  \left\{\begin{aligned}
  & \pd{h}{t} + \pd{}{x}\sum_{\alpha = 1}^N (h_{\alpha}u_{\alpha}) = 0, \\
  & \pd{}{t}(h_{\alpha}u_{\alpha}) + \pd{(h_{\alpha}u_{\alpha}^2)}{x} = -h_{\alpha}\pd{p}{x} + (m_{\alpha+1/2} - m_{\alpha-1/2}) + (\tau_{\alpha+1/2} - \tau_{\alpha-1/2}),
  \end{aligned}\right.
\end{equation}
where for each layer $\alpha = 1, \ldots, N$, $h_{\alpha}=\ell_{\alpha}h$ denotes the layer thickness, $\ell_{\alpha} > 0$ being a given constant, $\sum_{\alpha = 1}^N \ell_{\alpha} = 1$, and $u_{\alpha}(t,x)$ the averaged velocity in layer $\alpha$. The source terms $m_{\alpha+1/2}, ~ \tau_{\alpha+1/2}$ stand for the momentum exchange and the friction between layers $\alpha$ and $\alpha+1$ respectively. 

As proposed in
\cite{Audusse2011a},
we solve this multilayer shallow water (MLSW) system \eqref{eq:SWML} using a first-order finite volume scheme in which the numerical flux is built by a kinetic formulation. The friction term between the layers $\mathcal{\tau}_{\alpha+1/2}$ is discretized in an implicit way by
  \begin{equation}\label{eq:friction-MLSW}
  \tau_{\alpha+1/2} = \bar\delta^2 \dfrac{2(u_{\alpha+1}-u_{\alpha})}{h_{\alpha+1}+h_{\alpha}} \quad  \text{for} \quad 0 < \alpha < N, \qquad\tau_b = \tau_{1/2} = \bar\delta^2 \dfrac{2u_1}{h_1}, \quad\tau_{N+1/2} = 0.
  \end{equation}
  We notice that the third expression is due to the no-stress condition at the surface while the second one, expressing the bottom friction, is based on the no-slip condition and on a first-order expansion of the velocity. Using this model imposes some constraints on the vertical discretization: it requires a very thin layers near the bottom in order to accurately compute the friction while the velocity of two adjacent layers (in the viscous region) must not be too different in order to preserve the hyperbolicity of the model. This later condition implies that relative thickness $\ell_{\alpha}$ of the layers can be varied but only gradually. For the present study, a discretization such as $\ell_{\alpha} = z_{\alpha} - z_{\alpha-1}$ with
  \begin{equation*}
   z_{\alpha}:= \frac{e^{10\alpha/N} - 1}{e^{10}-1}\quad \text{for} \quad 0 \leq \alpha \leq N := 100
  \end{equation*}
  seems to give satisfactory results. Nevertheless the simulation with MLSW model is computationally expensive.

We interpret this MLSW system as a numerical scheme to solve  the RNSP equations \eqref{Eccl2}-\eqref{Eccl4}, the layers being the numerical discretization along the vertical direction  (see \cite{Ghigo2017} for a similar point of view in elastic tubes). 
In our approach, at the RNSP level, the two superposed layers have different physical properties (viscous/ ideal fluid), and analyzed through asymptotic rescalings of the equations. At the integrated level, the closest formulation is the apparent topography, 
we have again two superposed layers of different physical nature, and their relative thickness, $\bar\delta\delta_1/h$, is not fixed but evolves in time, in contrast with the multi-layer model.

  \section{Numerical illustrations}\label{Examples}
The aim of this last part is to give a few illustrations of the behaviour of the ESW system \eqref{eq:ESW-IBL}. We are aware that more accurate analysis is mandatory, in both the numerical approach (in particular numerical boundary conditions) and the qualitative and quantitative behaviour of the model. Fisrt we give a convergence study based on a steady-state solution. A second step is devoted to a comparison between the ESW solutions and those of the classical viscous shallow water system. Finally we compute the solutions over a small bump, in order to evidence the above mentioned phase-lag of the friction term, and the behaviour of the model with respect to various parameters.

In all test cases, we considered a very thin viscous layer by setting $\bar\delta = 10^{-3}$. 
The Froude number $Fr$ is set to unity meaning that the longitudinal velocity was scaled by the reference celerity $\sqrt{gh_0}$. The computation domain was $[0,L]$. Initial conditions was $h(0,x) = h_0, ~ u_e(0,x) = 1, ~ \delta_1(0,x) = 0 ~\forall x\in[0,L]$. We set $h_0 = 2$ for subcritical test cases, so $Fr_0 \simeq 0.7$, while we used $h_0 = 0.5$, i.e. $Fr_0 \simeq 1.4$, for supercritical cases. On the boundary conditions, we considered at $x=L$ a free outflow boundary. At $x=0$ for both sub- and supercritical flows we imposed a constant velocity $u_e = 1$ with flat profile, so that $\delta_1 = 0$. For supercritical flows, a constant water depth $h =h_0$ was imposed as well. In the case of subcritial flows, $h$ was computed using the Riemann invariant of classical shallow water model. This approach is only an approximation for ESW, more in-depth study on  numerical boundary condition is of course needed.

  \subsection{Convergence study on a Blasius-like steady solution}

In this section we investigate an approximate stationary solution of the ESW model which is an order one perturbation of a basic
stationary solution to the ESW system, namely $(h^0,u_e^0,\delta_1^0)$, where $h^0$ and $u_e^0$ are constant solutions
of the frictionless shallow water system, and $\delta_1^0$ is  the classical Blasius
profile for the von K\'arm\'an equation on a flat plate. 
We look for a solution  to the ESW system at first order in $\bar\delta$
\begin{equation}\label{eq:Blasius-linearized}
    h = h^0 + \bar\delta h^1, \quad u_e = u_e^0 + \bar\delta u_e^1, \quad \delta_1 = \delta_1^0 + \bar\delta\delta_1^1.
    \end{equation}
A very interesting feature of this solution is that the order one terms are not necessarily stationary.
This gives an explicit illustration of the actual interaction between the viscous layer and the perfect fluid.

Plugging the expansion \eqref{eq:Blasius-linearized} in \eqref{eq:ESW-IBL}, we recover through the first two equations a standard inviscid shallow water model, for which a basic stationary solution consists in constant $h^0$ and $u_e^0$. Now we turn to the stationary von K\'arm\'an equation 
\begin{equation*}
u_e\delta_1\partial_x u_e + \partial_x\left(\frac{u_e^2\delta_1}{H}\right) = \frac{f_2H}{\delta_1}u_e.
\end{equation*}
In this equation, $H$ and $f_2$ depend on 
$\Lambda_1 = \delta_1^2\partial_xu_e =  \bar\delta \delta_1^0\partial_x u_e^1 + O(\bar\delta^2)$.
Therefore at zeroth order $H$ and $f_2$ are constant, so that we indeed recover the classical Blasius solution
\begin{equation}\label{eq:Blasius-solution}
\delta_1^0 = \sqrt{\frac{2f_2 H^2}{u_e^0}x} = 1.718\sqrt{x}, \quad \bar\tau_b^0 = \frac{f_2 H}{\delta_1^0} = \frac{0.332}{\sqrt{x}}.
\end{equation}
We use the solution $(h^0,u_e^0,\delta_1^0)$ as the basic solution in the expansion \eqref{eq:Blasius-linearized}, and turn now to order one terms. Once again, straightforward computations lead to uncoupling the first two equations, yealding the
following linearized shallow water model:
\begin{equation}
    \partial_t\begin{pmatrix}h^1\\u_e^1\end{pmatrix} 
        + \begin{pmatrix}u_e^0 & h^0\\\frac{1}{Fr^2} & u_e^0\end{pmatrix}\partial_x\begin{pmatrix}h^1\\u_e^1\end{pmatrix}
        = \begin{pmatrix}\partial_x(\delta_1^0u_e^0)\\0\end{pmatrix}.
\end{equation}
Notice that this system has a stationary solution, given by
\begin{equation*}
h = h^0 + \bar\delta \frac{Fr_0^2}{Fr_0^2 -1}\delta_1^0, \quad u_e = u_e^0 + \bar\delta \frac{1}{1-Fr_0^2}\delta_1^0,
\end{equation*}
in which we have defined the local Froude number $Fr_0=Fr\ u_e^0/\sqrt{h^0}$. 

We consider the computation domain $x\in[0,0.1]$.
On figure \ref{fig:convergence-blasius}-left we plot the Blasius solution in which the displacement thickness $\delta_1^0$ increases in function of $\sqrt{x}$ while the friction decreases from infinity according to \eqref{eq:Blasius-solution}. On Figure \ref{fig:convergence-blasius}-right we display the results of a mesh convergence study on the gap between $\delta_1$ and its zeroth order approximation $\delta_1^0$. The convergence study was performed in both sub- and supercritical regimes. 

First we remark 
that when the mesh size $\Delta x$ is small enough, namely $\Delta x \leq 10^{-4}$, numerical results reach the model error. Indeed, we obtained at this spatial resolution that 
\begin{equation*}
\int_0^{0.1}|\delta_1 - \delta_1^0|dx \simeq 0.1 \bar\delta \quad\text{and so}\quad \delta_1 - \delta_1^0 = O(\bar\delta).
\end{equation*}

Next, we notice that the supercritical case converges faster than the subcritical one. This could be explained by the fact that, on the one hand, numerical treatment of the left boundary condition is more accurate in the supercritical case as we have noticed before; on the other hand, it is well known that the HLL-type numerical flux \eqref{eq:numerical-fluxes} is also more accurate in that case.
\begin{figure}[ht!] 
      \centering
      \includegraphics[width=0.49\linewidth]{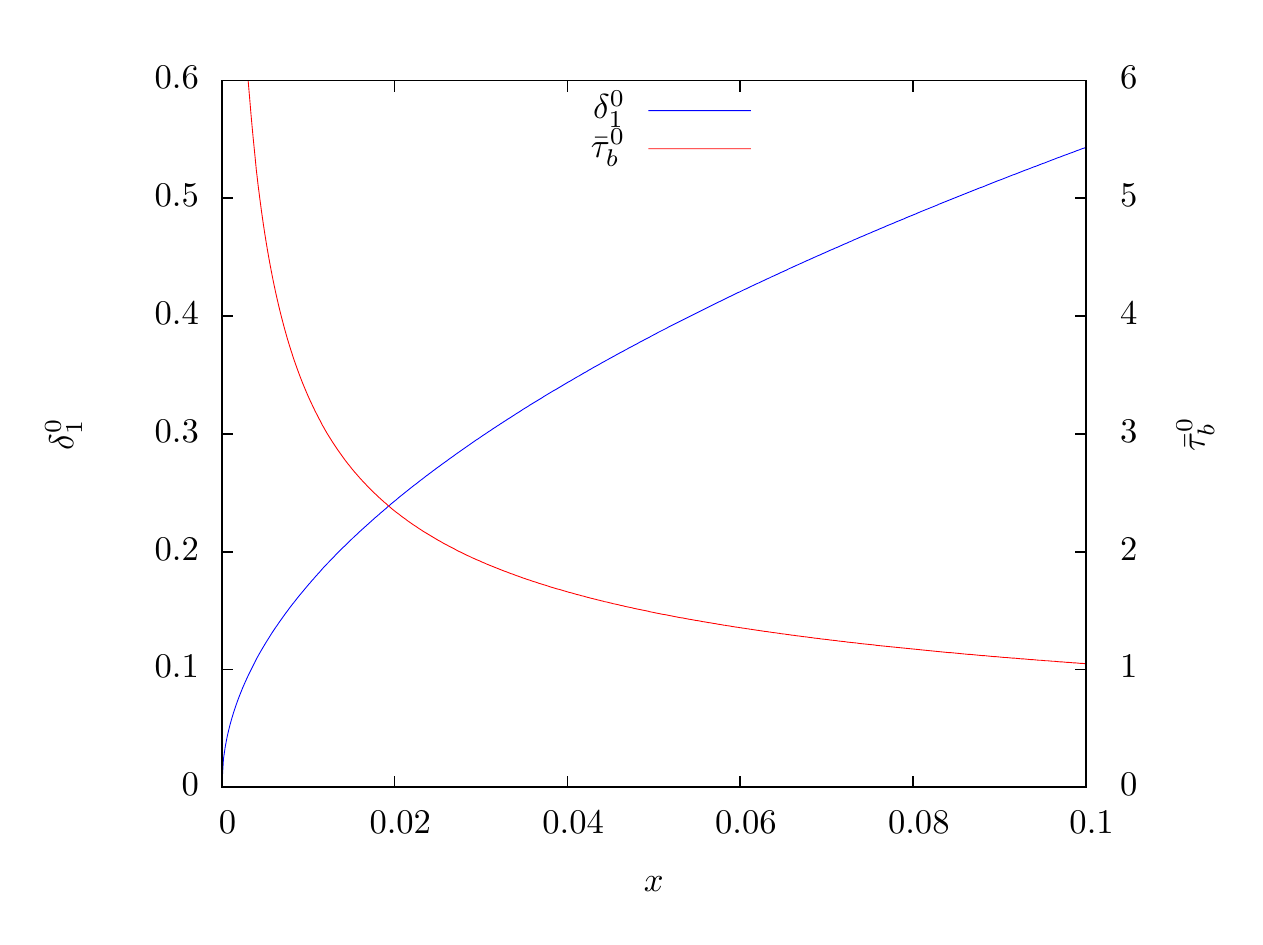}
      \includegraphics[width=0.49\linewidth]{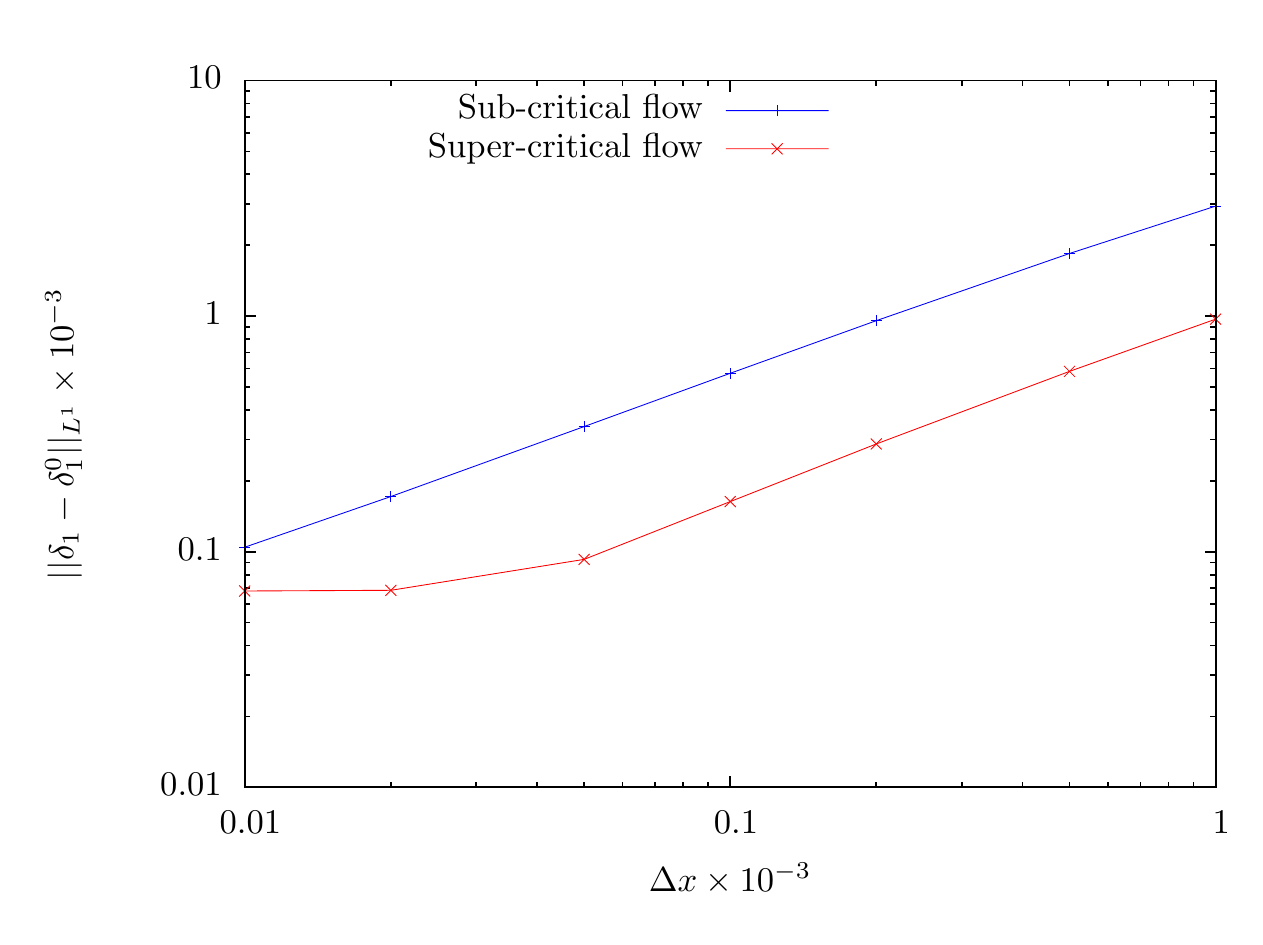}
      \caption{Left: Blasius solution for sake of illustration, displacement thickness;
      $\delta_1^0 =   1.718\sqrt{x}$,
      and shear $\bar\tau_b^0 ={0.332}/{\sqrt{x}}$.
      Right: error $\int_0^{0.1}|\delta_1-\delta_1^0| dx$ as the function of the mesh size $\Delta x$ }\label{fig:convergence-blasius}
    \end{figure}

\subsection{Impulsively started flow over a flat bed} \label{sec:blasius-stokes}
We turn now to a configuration introduced by Stewartson \cite[Sec. 3]{Stewartson1951,Stewartson1973} as a simple test-case to study unsteady boundary layer solutions. It consists in a semi-infinite flow impulsively started from rest at $t=0$ with constant velocity $u_e$, see figure \ref{Fig-test1}. 
The fluid is injected continuously at $x=0$ with a constant velocity, the flow must satisfy the no slip boundary condition for $x>0$. 

The solution exhibits different behaviours depending on two asymptotic regimes: for large $t$ (or small $x$) we recover the Blasius solution \eqref{eq:Blasius-solution}; conversely, for large $x$ (or small $t$), convective terms in the momentum equation \eqref{Eccl2} are negligible so the Prandtl system reduces to Stokes' first problem (also called Rayleigh problem by Stewartson). The solution for the velocity profile can be expressed using the erf error function:
\begin{equation}\label{eq:Stokes-solution}
  \begin{aligned}
      & \frac{\bar u}{u_e}=\textrm{erf}\left(\frac{\bar y}{2 \sqrt{t}}\right),\quad
  \delta_1 = 2 \sqrt{\frac{t}{\pi}},\quad
  \bar\tau_b = \frac{1}{\sqrt{\pi t}},  \\
  & H = 1 + \sqrt{2} \simeq 2.414, \quad f_2 = \frac{2}{\pi(1+\sqrt{2})} \simeq 0.264. 
  \end{aligned}
  \end{equation}
  Transition between these two solutions occurs for $x=O(t)$.

  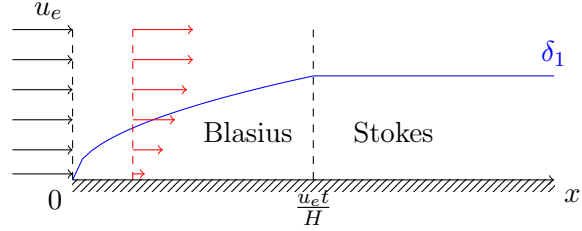
\begin{figure}[ht!]\centering
    \begin{tikzpicture}[domain=-1:8, scale=0.8]
      \draw[->] (0,0) node[below left]{$0$} -- (8,0) node[below right]{$x$};
      \fill[pattern=north east lines] (0,0) -- (8,0) -- (8,-0.2) -- (0,-0.2);
      \draw[black,dashed] (0,0) -- (0,2.5);
      \draw[blue] plot[domain=0:4] (\x,{5*(1.73*sqrt(0.01*\x))});
      \draw[blue] (4,1.73) -- (8,1.73) node[above]{$\delta_1$};
      \draw[->] (-1,0.1) -- (0,0.1);
      \draw[->] (-1,0.5) -- (0,0.5);
      \draw[->] (-1,1) -- (0,1.);
      \draw[->] (-1,1.5) -- (0,1.5);
      \draw[->] (-1,2.) -- (0,2);
      \draw[->] (-1,2.5) -- (0,2.5) node[above left]{$u_e$};
      \draw[thin, dashed, red] (1,0) -- (1,2.5);
      \draw[->, red] (1,0.1) -- (1.2,0.1);
      \draw[->, red] (1,0.5) -- (1.5,0.5);
      \draw[->, red] (1,1) -- (1.7,1);
      \draw[->, red] (1,1.5) -- (1.9,1.5);
      \draw[->, red] (1,2) -- (2,2);
      \draw[->, red] (1,2.5) -- (2,2.5);
      \draw (2.,0.8) node[right] {Blasius};
      \draw[black,dashed] (4,0) node[below]{$\frac{u_e t}{H}$} -- (4,2.5);
      \draw (4.5,0.8) node[right] {Stokes};
    \end{tikzpicture}
    \caption{{\small Impulsively started flow over a flat bed (Stewartson's 1951-1973 problem): transition between unsteady Stokes Rayleigh problem and steady Blasius problem.}}
    \label{Fig-test1} 
  \end{figure}

  Stewartson noticed that an approximate integral form can be used to solve this problem. Assuming a {\em fixed profile}, e.g. Blasius (constant) value of $H$ and $f_2$, von K\'arm\'an equation (\ref{VKdelta1}) can be rewritten in the form 
  \begin{equation*}
    \partial_t(\delta_1^2)+\frac{u_e}{H}\partial_x(\delta_1^2)=2f_2H,
  \end{equation*}
  and it has to be complemented with the following initial and boundary value conditions:
  \begin{equation*}
    \delta_1(0,x) = \delta_1(t,0)= 0, \quad x,t\in \mathbb{R}^+.
  \end{equation*}
  The solution is readily obtained by the method of characteristics:
  \begin{equation}\label{Blasius-Stoke-delta1}
    \delta_1 = \left\{\begin{aligned}
    & \sqrt{\frac{2f_2 H^2x}{u_e}} & \text{ if } x \leq \frac{u_e t}{H}, \\
    & \sqrt{2f_2 H t} & \text{otherwise},
    \end{aligned}\right. 
    \qquad \text{and} \qquad
    \bar\tau_b = \left\{\begin{aligned}
    & \sqrt{\frac{f_2u_e^3}{2x}} & \text{ if } x \leq \frac{u_e t}{H}, \\
    & \sqrt{\frac{f_2 H u_e^2}{2t}} & \text{otherwise},
    \end{aligned}\right.
  \end{equation}
  Transition zone is found at the characteristic line  $x = \frac{u_e t}{H}$, i.e. viscous layer informations propagate at velocity $u_e/H$ as we have seen in equation \eqref{eq:FS-lambda3}. It is worth noticing that this solution with constant $u_e$ creates an unbounded viscous layer, without regards of the limitation of the water depth. Hence it is clearly physically not valid for large $x$ or $t$. 
  
  On figure \ref{fig:dd1}-left, we plot on the left side the evolution of the displacement $\delta_1$ along $x$ at different times. The two-regimes behaviour of the unsteady solution \eqref{Blasius-Stoke-delta1} is qualitatively well recovered: $\delta_1$ is constant in $x$ far from the entrance and increases in time; it reaches the Blasius solution which is steady. But, quantitatively, we observe a gap between the solution of ESW and the Stokes one, which increases in time. This
  is due to the choice of a fixed Blasius profile in the viscous layer. Regarding equations \eqref{eq:Blasius-solution} and \eqref{eq:Stokes-solution}, the Blasius and Stokes  profile have not exactly the same shape and this leads to different values of $H$ and $f_2$. More precisely, in the Stokes region the model predicts $\delta_1 = 1.067\sqrt{t}$ while the exact solution is $\delta_1 = 1.128\sqrt{t}$.  
   \begin{figure}[ht!]
    \centering
    \includegraphics[width=0.49\linewidth]{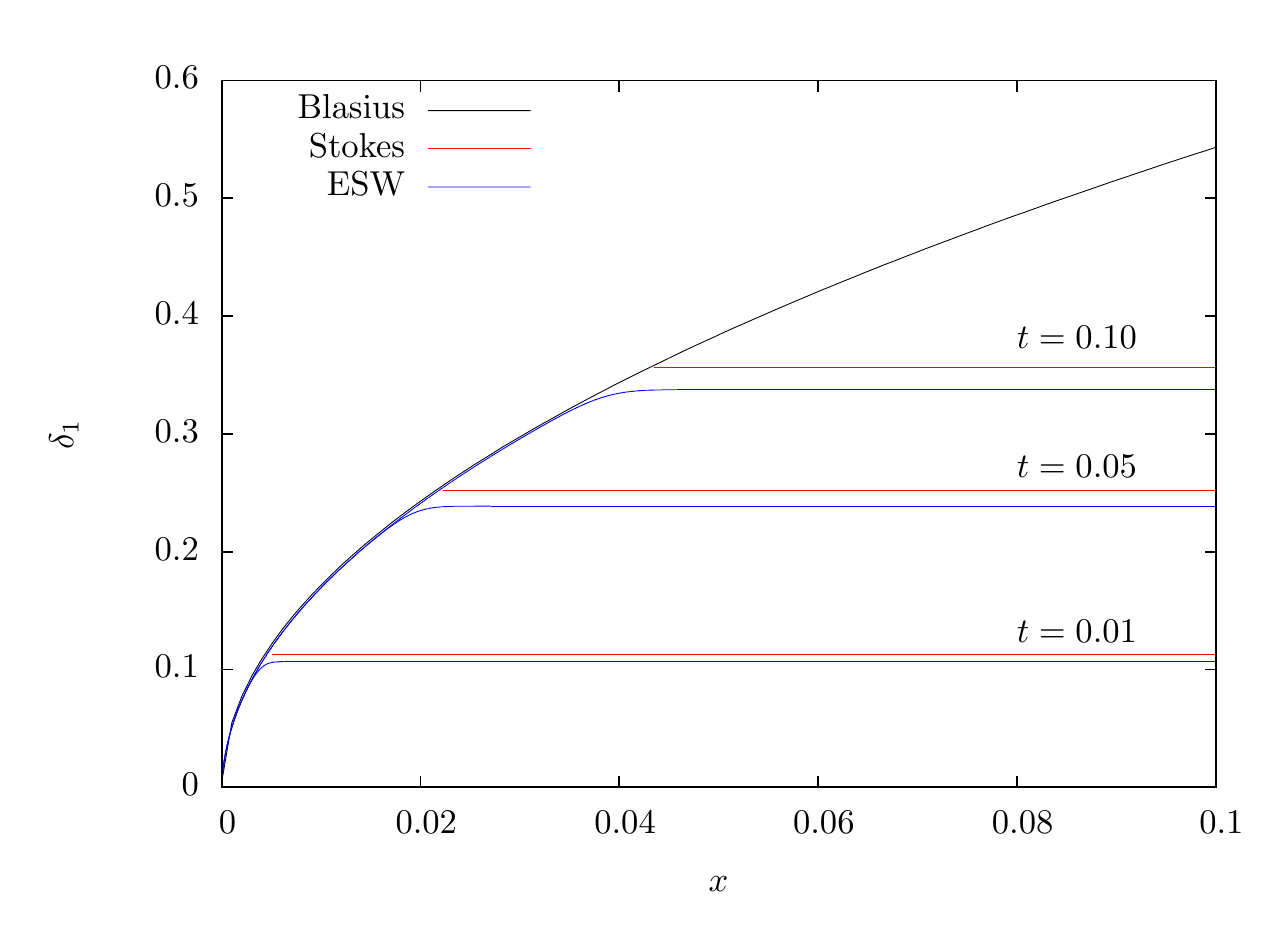} 
    \includegraphics[width=0.49\linewidth]{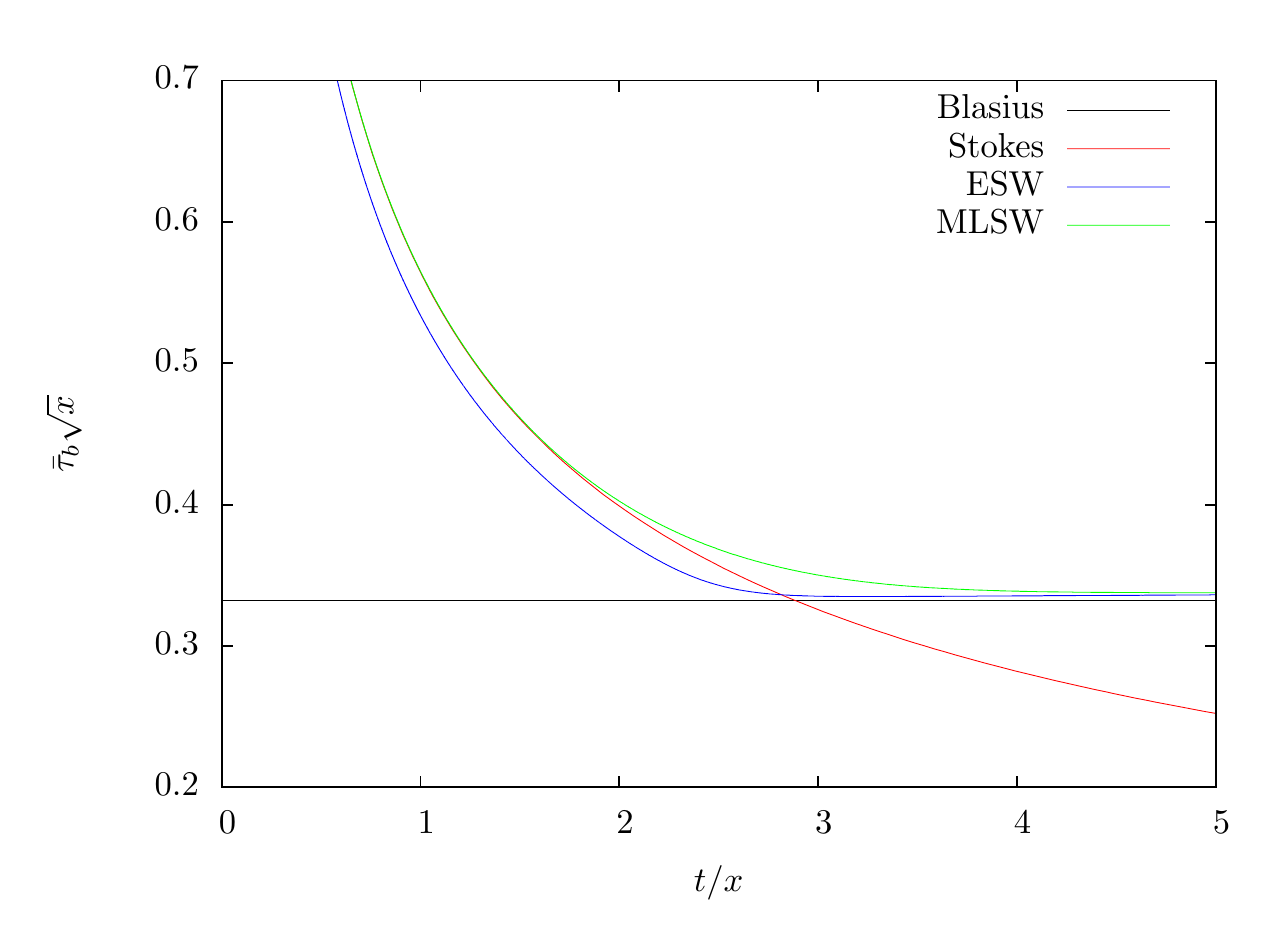} 
    \caption{{\small From Stokes to Blasius solution. Left: development of displacement thickness $\delta_1$ function of $x$ at different times; Right: comparaison of ESW and multilayer shallow water model (MLSW) for $\bar\tau_b\sqrt{x}$ in function of $t/x$: for small $t/x$, the solution is that of Stokes $\bar\tau_b\sqrt{x} = \sqrt{x/\pi t}$; at large $t/x$ the curves collapse on the Blasius solution $\bar\tau_b \sqrt{x} = 0.332$.}}\label{fig:dd1}
  \end{figure}

 To illustrate the transition in profiles from Stokes to Blasius, we propose here to compare 
the numerical results given by the ESW model with those of the multi-layer scheme described in Section \ref{sec:MLF}
\eqref{LWmomentum}--\eqref{eq:no-slip}. 
  On figure \ref{fig:dd1}-right, we plot the difference between the numerical solutions and the two asymptotic solutions as function of $t/x$. We plot on this figure $\bar\tau_b\sqrt{x}$ which is $\sqrt{x/\pi t}$ for small $t/x$ (Stokes solution) and becomes $0.322$ (the Blasius value) for large $t/x$. On the solution of ESW, the unsteady part of ESW solution presents a small difference as well. This is simply due to the fact that the Blasius profile and the Stokes one have not exactly the shame shape as we have observed for $\delta_1$. The value $f_2H$ given by the Stokes solution is greater than that of the Blasius solution. Good agreement was found on the solution of MLSW: the solution tends from Stokes to Blasius behaviours; the transition between these two regimes is clearly captured.

  \subsection{Flows over a small bump}
  This final section considers cases with a  flat bottom with a bump.
  The lenght of the bump is such that those compatible with classical shallow water model.
    Hence, we consider now a   domain  $x \in [0,2]$ with a small bump of Gaussian form at the center:
  \begin{equation*}
      f_b = \alpha e^{-\frac{(x-1)^2}{2\sigma^2}}, 
  \end{equation*}
 with $\alpha$ and $\sigma$ given.
  Falkner-Skan closure, which includes as well the Blasius one, is used to compute the shape and friction factors. As defined by \eqref{eq:Falkner-Skan-closure}, we have to compute the pressure gradient parameter $\Lambda_1 = \delta_1^2\partial_x u_e$ for which accurate approximation for the partial derivative $\partial_x u_e$ is required in order to provide a faithful representation of the shape of thin viscous layer. Therefore, we test second order and fourth-order finite difference derivative:
  \begin{equation}\label{eq:velocity-gradient-estimation}
      (\partial_x u_e)_j = \frac{(u_e)_{j-2} - 8(u_e)_{j-1} +  8(u_e)_{j+1} - (u_e)_{j+2}}{12\Delta x} + o(\Delta x^4).
  \end{equation}
  
  \medskip
  \noindent{\bf Influence of the flow regime.}
  We use first $\alpha=0.01$ and $\sigma=0.1$. As in the test case of an impulsively started flow, the displacement thickness develops from zero, at $t=0$. It reaches a steady value when the time is large enough, typically when $t > xH/u_e \simeq 6$ from characteristic solution \eqref{Blasius-Stoke-delta1} applied for this case. However, this steady solution is no longer Blasius but is slightly perturbed due to the presence of the bump; this perturbation depends furthermore on the flow regime, as one can observe on figure \ref{fig:bump-FS-delta1-taub}.
  
  For sub-critical case, the flow is accelerated on the upstream side of the bump while it is decelerated on the downstream side. As a consequence, the displacement thickness decreases before the crest and increases after. This behaviour is also well reported by classical shallow water model, even with linearized solution \eqref{eq:lin-sol-sw}. However, ESW provides an asymmetric friction due to inertia effect of the fluid, compared with \eqref{eq:lin-sol-sw}. More precisely, the friction reaches its maximum before the crest for sub-critical flow while for  super-critical  case, it becomes maximum after the crest, see figure \ref{fig:bump-FS-delta1-taub} (right). This is in fact the important phage-lag behaviours that we have found to cope with ESW. 
  
 \begin{figure}[ht!] 
      \centering
      \includegraphics[width=0.49\linewidth]{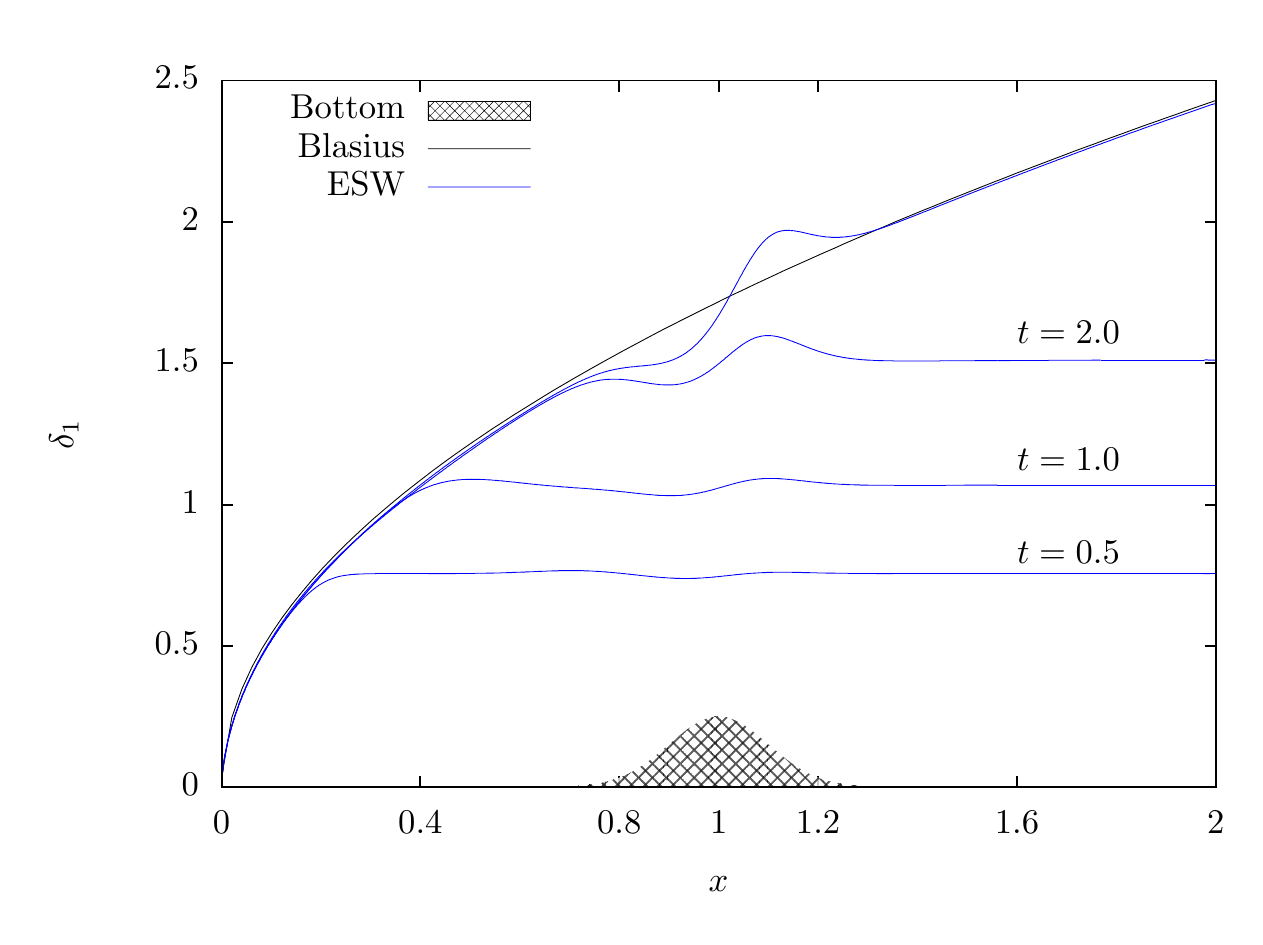}
      \includegraphics[width=0.49\linewidth]{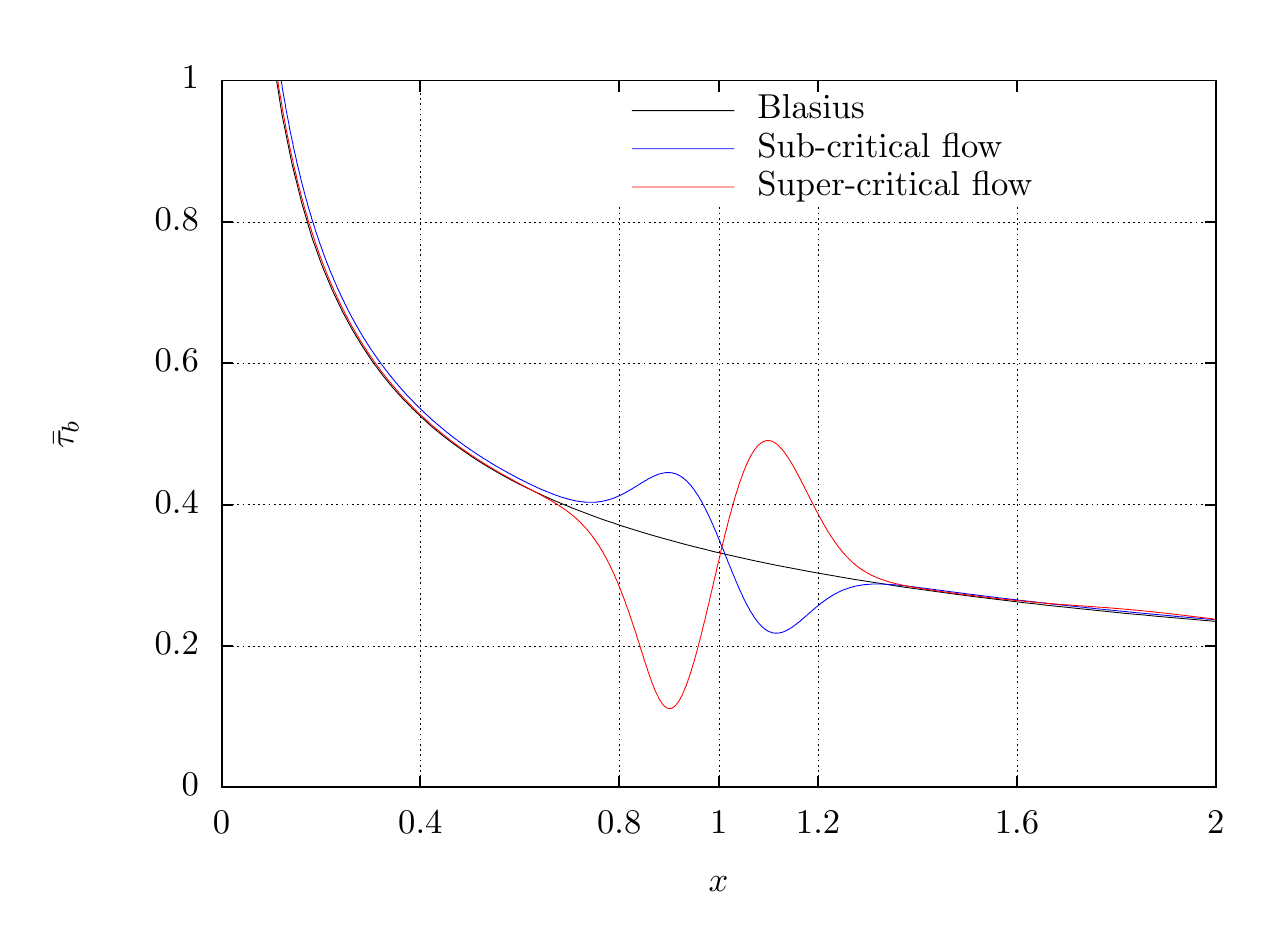}
      \caption{Flow over a small bump at different times. Left: development of displacement thickness at different times in sub-critical case. Right: behaviours of friction in sub-critical and super-critical cases at long time ($t=6$)}\label{fig:bump-FS-delta1-taub}
    \end{figure}
  
   \medskip
  \noindent{\bf Phase-lag reduced for shorter bump.} 
 Here we investigate the influence of the length of the bump on the friction computed with ESW model. Returning to sub-critical case, we perform now the same test case but with a shorter bump, by imposing $\sigma = 0.05$. Regarding the result given with Falkner-Skan closure, i.e. when the shape factor depends on velocity gradient $\partial_x u_e$, figure  \ref{fig:bump-FS-taub-sigma} (left) shows that the variation of friction is more important than the case with $\sigma = 0.1$. The phage-lag   is always observed.  
  \begin{figure}[ht!] 
      \centering
      \includegraphics[width=0.49\linewidth]{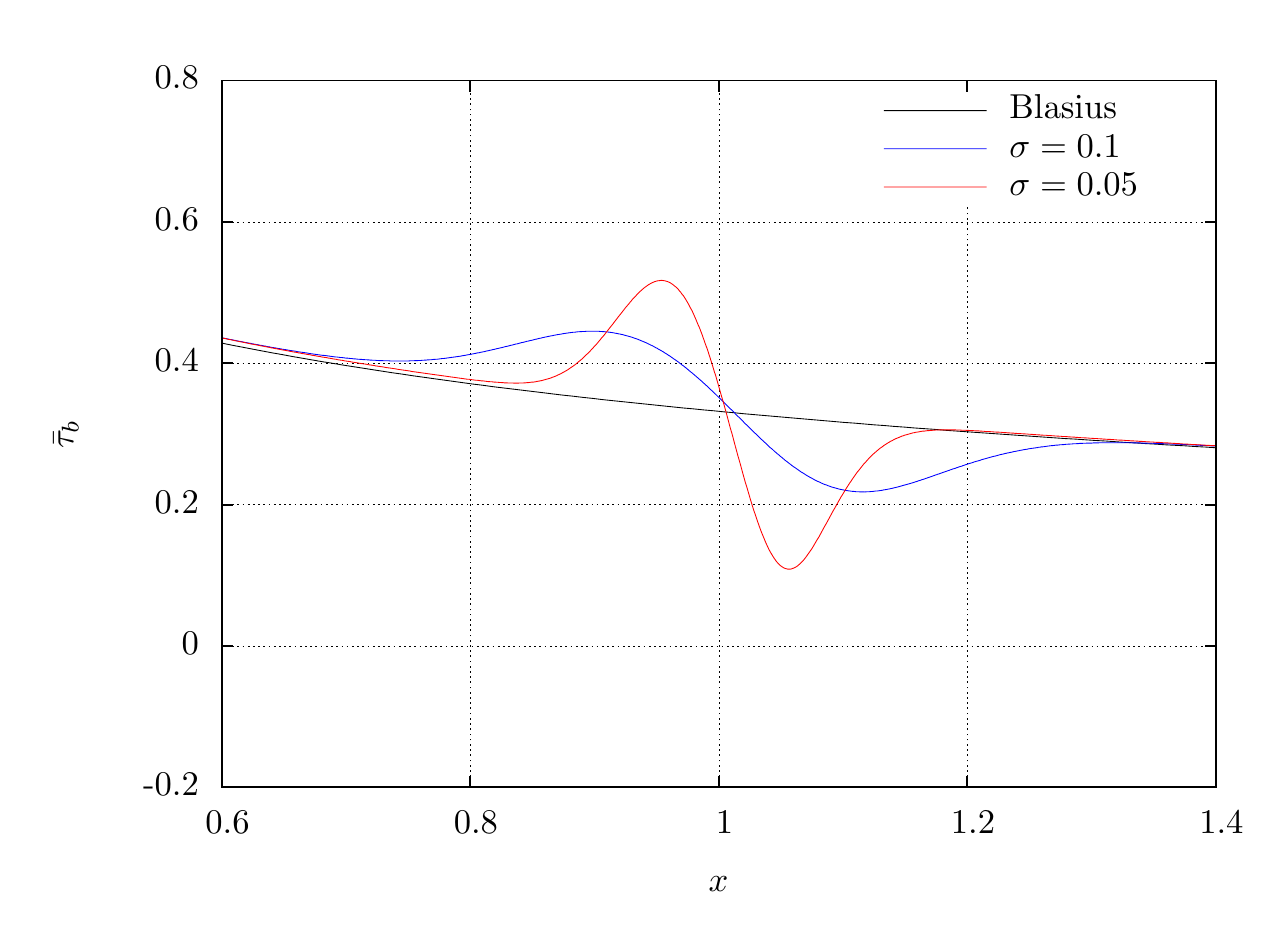}
      \includegraphics[width=0.49\linewidth]{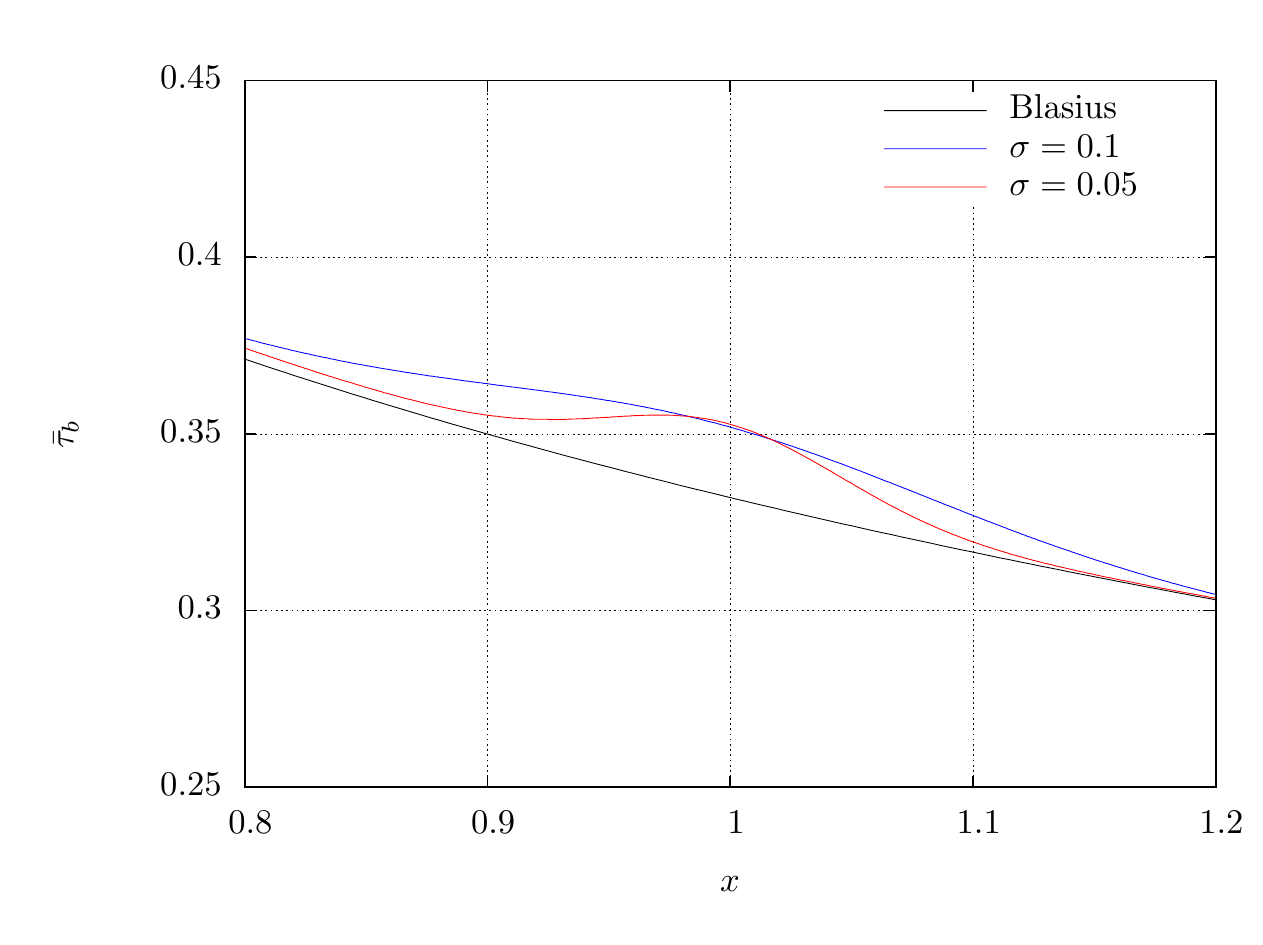}
      \caption{Phase-lag for two short bumps ($\alpha=0.05$) compared to Blasius solution ($\alpha=0.$) . Left: friction with Falkner-Skan closure (\ref{eq:Falkner-Skan-closure}). Right: friction with constant value of $H=2.59$, and $f_2=0.22$ from Blasius solution.}\label{fig:bump-FS-taub-sigma}
    \end{figure}
    
 To highlight the role of the closure imposed in viscous layer, we 
 do again this test but with a constant value of shape factor (e.g. by using Blasius closure). 
 It can be observed on figure \ref{fig:bump-FS-taub-sigma} (right) that both amplitude and phase-lag of friction are significantly reduced compared to the case with (variable) Falkner-Skan closure. Moreover, constant shape factor does not allow the friction to decrease enough in decelerated region localized at downstream side of the bump. As a consequence, this kind of closure is unable to recover reverse flow whatever the bump shape.  
  
 \medskip
 \noindent{\bf Larger friction for higher bump, reverse flow observation.} This last test case is devoted to highlight  the possibility   of ESW to  capture indeed reverse flow if the bump is high enough.
 We fix the bump length to be $\sigma = 0.1$ and change  $\alpha$. Falkner-Skan closure (\ref{eq:Falkner-Skan-closure}) is of course imposed. Figure \ref{fig:bump-FS-taub-alpha-order} shows the displacement thickness and the friction for increasing $\alpha$ up to 
 negative value of friction (which defines steady boundary layer separation). Associated to boundary layer separation is an increased displacement thickness.

 Note that accurate estimation of velocity gradient $\partial_x u_e$ is very important to capture this particular phenomena. That is why   we have used the fourth-order formula \eqref{eq:velocity-gradient-estimation}. Using a second-order approximation for $\partial_x u_e$ results just in the incipient  separation.
 
  \begin{figure}[ht!]   
    \centering
     \includegraphics[width=0.49\linewidth]{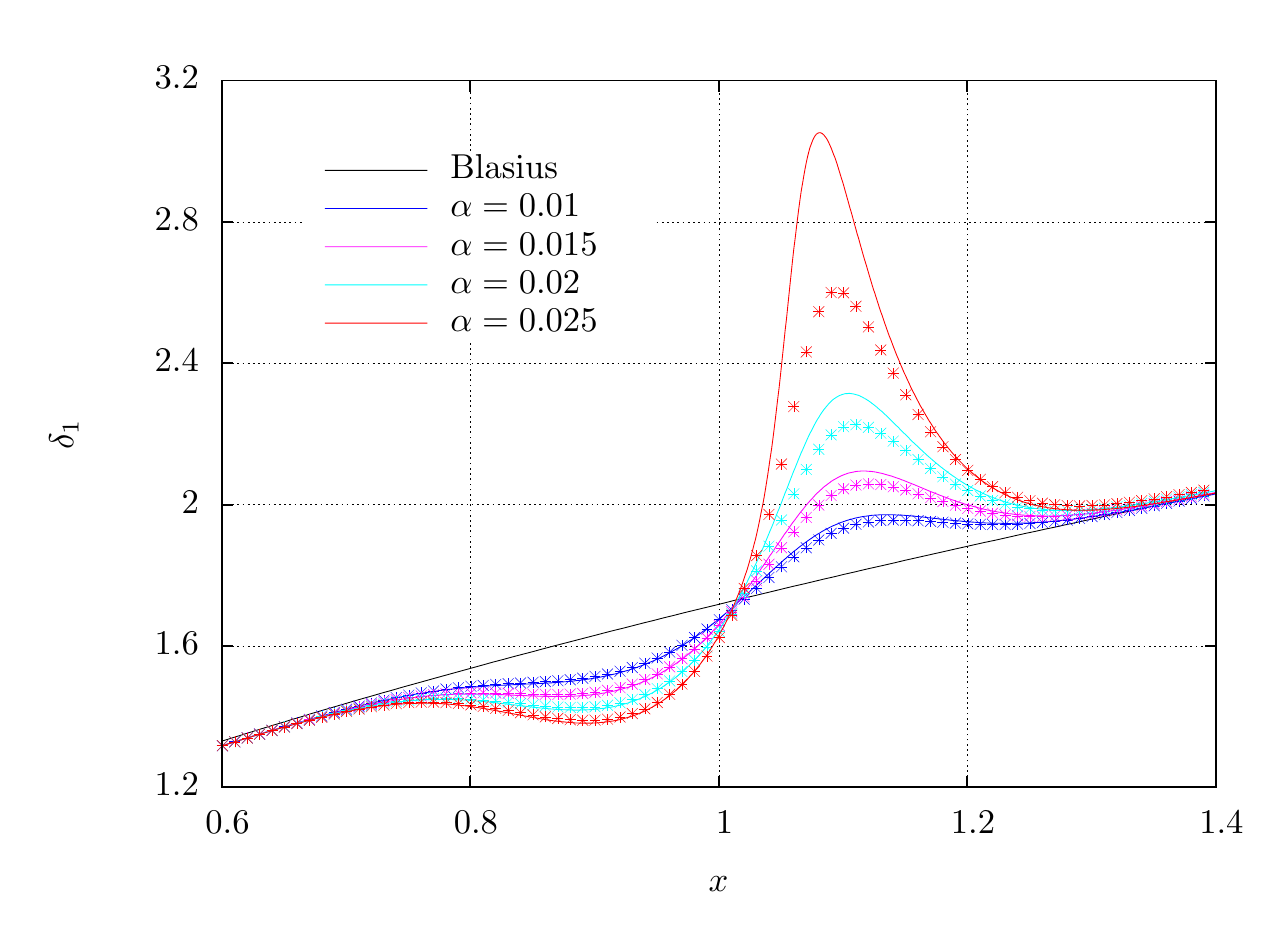}
     \includegraphics[width=0.49\linewidth]{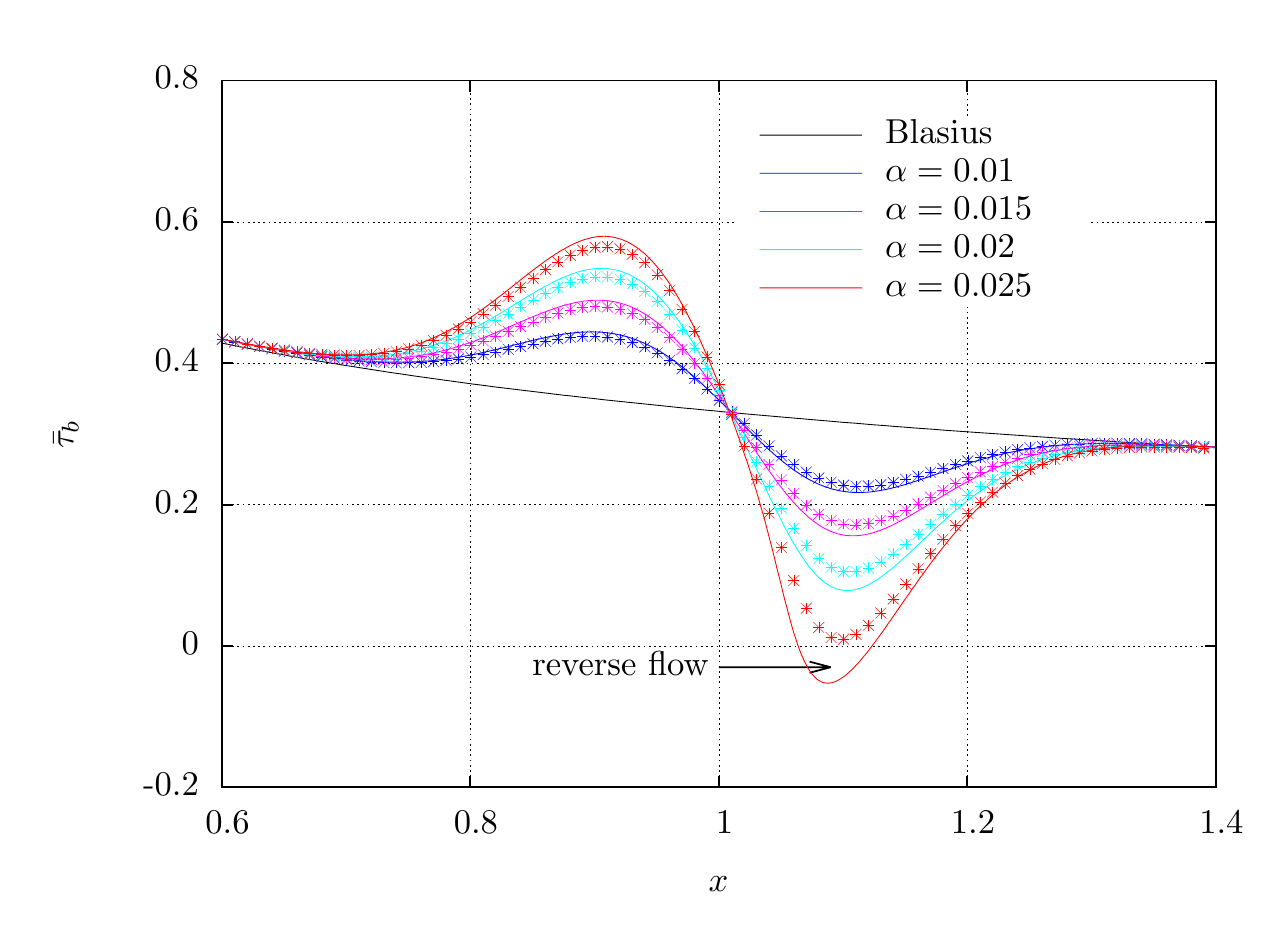}
      \caption{ESW for various bump height $\alpha$ at fixed lenght $\sigma=0.1$ with 2nd order (symbols) and 4th order (lines) discretization.
      Left the displacement thickness, right the wall shear stress. The influence of the 2nd and 4th order derivative is noticeable for  $
      \alpha = 0.025$ near separation. }
     \label{fig:bump-FS-taub-alpha-order}
    \end{figure}

   \begin{figure}[ht!]   
    \centering
      \includegraphics[width=0.49\linewidth]{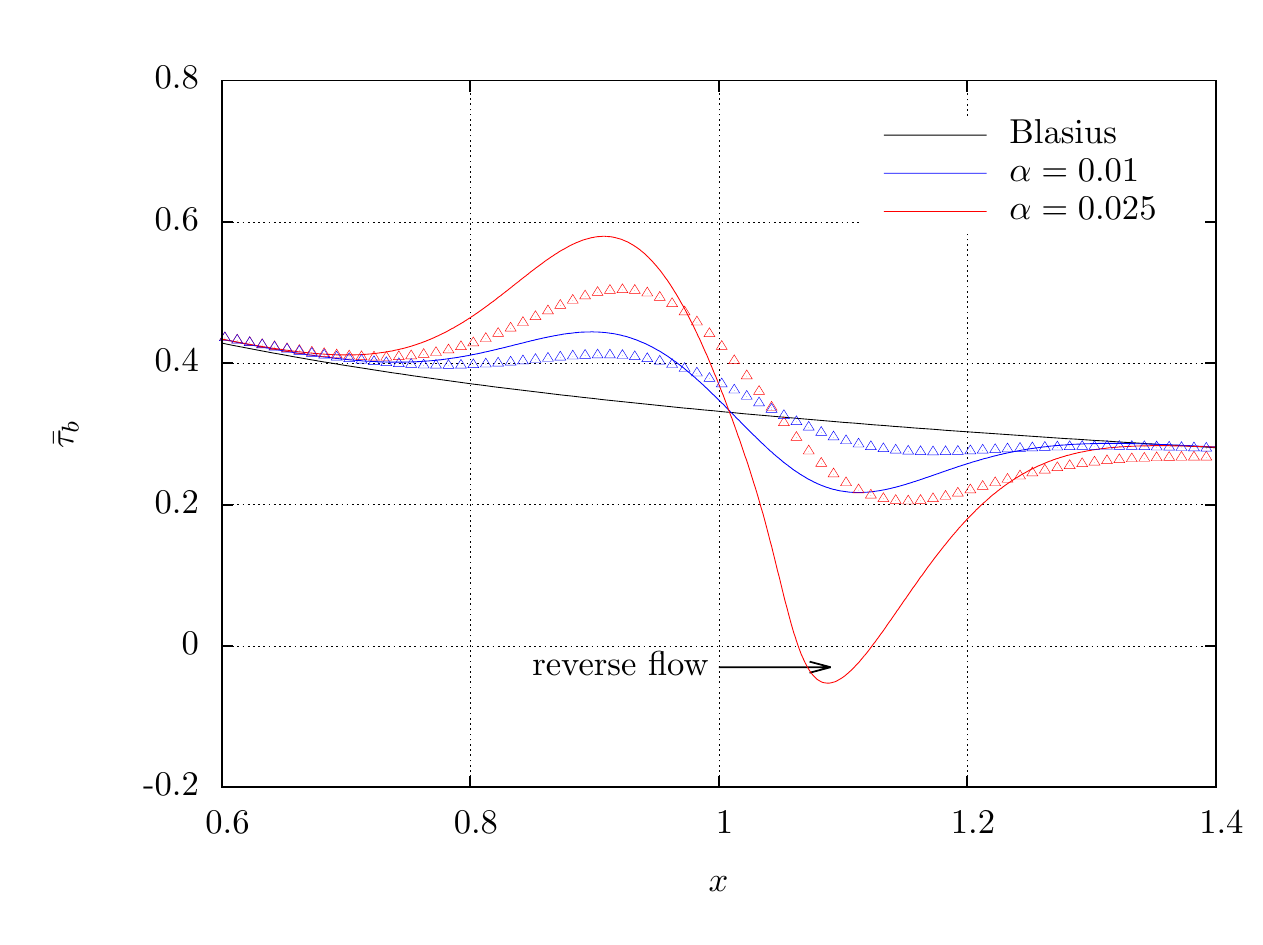}
      \includegraphics[width=0.49\linewidth]{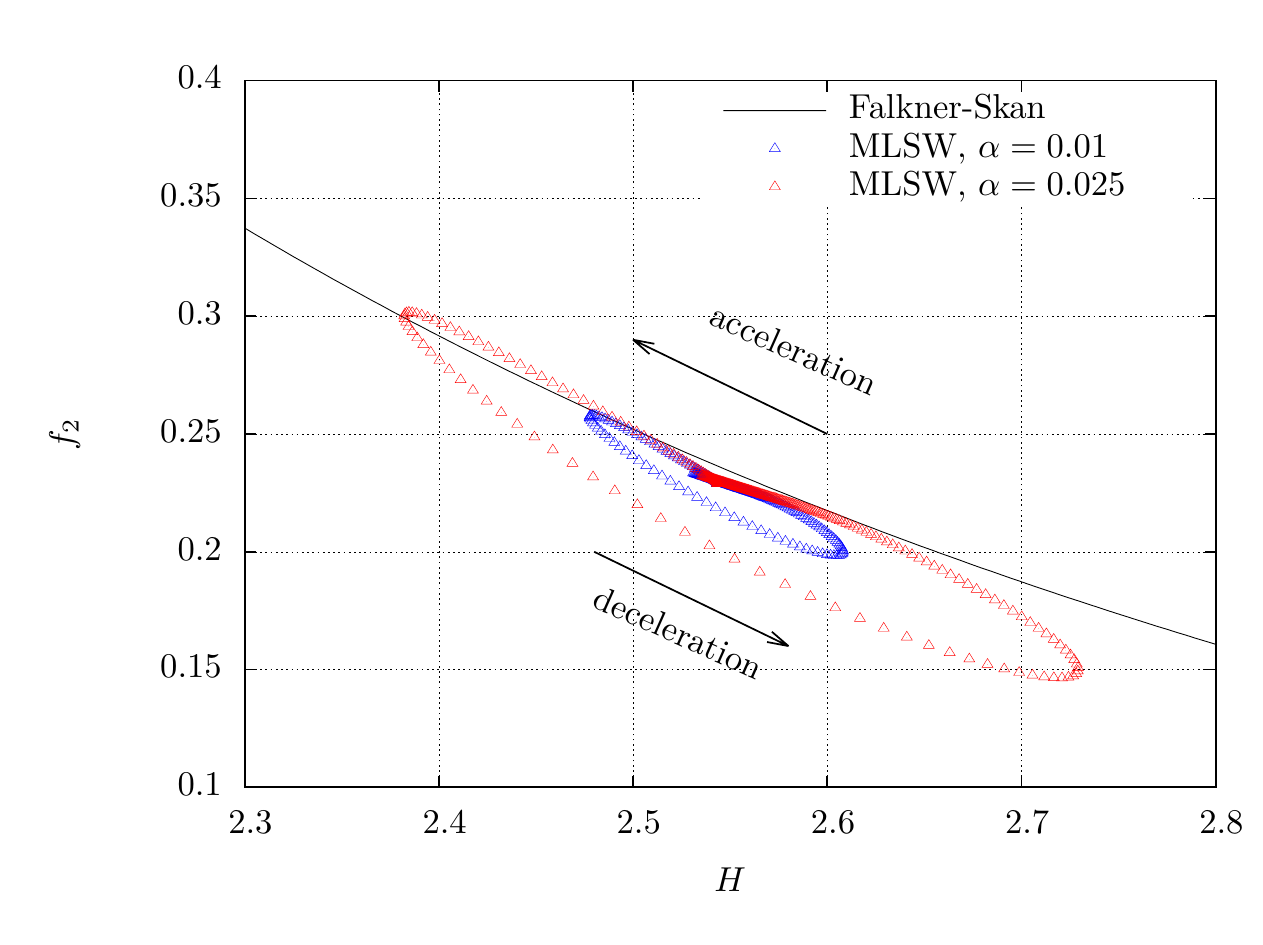}
      \caption{Friction of ESW (lines) vs MLSW (symbols) for two values of $\alpha$.
      Right: relation between $H$ and $f_2$ given by MLSW (symbols) compared to Falkner-Skan closure used for ESW (line). }
      \label{fig:bump-FS-taub-alpha}
    \end{figure}
 
 Finally,  performing this test case with MLSW model allows to find as well the phase-lag, but the amplitude of friction is smaller as observed on figure \ref {fig:bump-FS-taub-alpha}. It seems that numerical solver of MLSW model together with the first-order approximation \eqref{eq:friction-MLSW} on $\bar\tau_b$, and even with 100 layers, is not enough to accurately capture the friction. 
 
 Nevertheless,   plotting the friction factor $f_2$ as  function of the shape factor $H$ and comparing it with Falkner-Skan closure (\ref{eq:Falkner-Skan-closure}) shows that indeed, the reduced shear and shape factors computed from the profilers of MLSW are close to Falkner Skan curve. The smaller the angle, the closer the curves.
 The agreement is better during the accelerated phase of the flow.  
    

  \section{Conclusion}

 In  order to improve "shallow water" models, 
this paper proposed a novel description of parietal friction for
free surface shallow flows (long wave approximation)  in large Reynolds number  limit. 
The proposed  model relies on a perfect fluid -- viscous layer decomposition. It  consists in a system of four equations: conservation of mass (\ref{eq:sw-mass}),  momentum  (\ref{eq:sw-momentum-final}), an ideal fluid equation (\ref{eq:FP}), and the von K\'arm\'an equation (\ref{VKprof}). 
  The first two are similar to  classical shallow water system with a slight correction  on  momentum flux, and with a specific friction term.
  
  Two equivalent versions of the model can be obtained by interpreting them under two points of view. In the first one, the viscous layer acts as a new topography for the model, see (\ref{eq:ESW-apparent-topo}). In the second one,
this is  an example of ``Interactive Boundary Layer'' or ``Viscous Inviscid Interaction'', see \eqref{eq:ESW-IBL}, and it seems to be the most convenient choice for numerical purposes. 
  In these models the friction term is no longer an empirical combination of velocity and depth (as in usual laws such as  Darcy or Manning) but the result of a viscous layer like approach. A crucial point at this stage is the choice of  an apropriate closure for the shape factor and the friction factor in the viscous layer, in order to obtain a closed system after integration.
  
  As it is, the friction term actually depends on the topography, as evidenced by the examples provided in the last part of this work. In particular its maximum is reached before or after the summit of a bump, depending on the criticity of the flow. Also, possible boundary layer separation with recirculation  
  downstream (in subcritical case) of the bump can be observed as in similar cases in litterature (\cite{Lagree2005a,Lagree2005,Lagree2007} in case of flows in pipes, and as in preliminary comparisons with multilayer shallow water  \cite{Audusse2011a}).
  Those two behaviors are impossible to observe in the classical Shallow Water model.  
  
  Our proposed  approach is restricted to laminar flows, but the ideas developed here can be extended with little modifications to mean turbulent profiles.
  
  The main drawback of this model is the viscous layer/ideal fluid decomposition, which forbids  the viscous layer to fill all the water depth 
  far downstream of the bump. 
  Extra modelling in this direction, as well as a more careful study of the closure laws in the viscous layer have now to be worked out and tested. Also, the system is conditionnally hyperbolic, and the numerical study has to be improved, in particular boundary conditions. In spite of these limitations, the model is a good compromise between the classical shallow water system and RNSP equations (discretized with the Multilayer Saint-Venant scheme\cite{Audusse2011a}), much more costly in computation time.

  \section*{Acknowledgments}
  We would like to thank
  Nicole Goutal (EDF R\&D, Laboratoire d'Hydraulique Saint-Venant) for insightful discussions.


\end{document}